\documentclass[12pt]{amsart}

\usepackage{graphicx, amssymb, mathrsfs}
\usepackage[all]{xy}
\usepackage{enumerate}

\newtheorem{theorem}{Theorem}[section]
\newtheorem{proposition}[theorem]{Proposition}
\newtheorem{lemma}[theorem]{Lemma}
\newtheorem{corollary}[theorem]{Corollary}

\theoremstyle{definition}
\newtheorem{definition}[theorem]{Definition}

\newtheorem{remark}[theorem]{Remark}

\newcommand{\C}{\mathbb{C}}

\newcommand{\OO}{\mathscr{O}}
\newcommand{\CC}{\mathscr{C}}
\newcommand{\LL}{\mathcal{L}}
\newcommand{\R}{\mathbb{R}}

\newcommand{\N}{\mathbb{N}}

\newcommand{\Aut}{\operatorname{Aut}}
\newcommand{\id}{\operatorname{id}}
\renewcommand{\div}{\operatorname{div}}

\title{A soft Oka principle for proper holomorphic embeddings of open Riemann surfaces into $(\C^*)^2$}

\author{Tyson Ritter}
\address{Matematisk Institutt, Universitetet i Oslo, Postboks 1053 Blindern, 0316 Oslo, Norway} 
\email{tysonr@math.uio.no}



\date{26 March 2014. Latest minor changes 6 May 2014.}

\thanks{The author was supported by Norwegian Research Council grant NFR-209751/F20.}

\begin{document}

\begin{abstract}
Let $X$ be an open Riemann surface. We prove an Oka property on the approximation and interpolation of continuous maps $X \to (\C^*)^2$ by proper holomorphic embeddings, provided that we permit a smooth deformation of the complex structure on $X$ outside a certain set. This generalises and strengthens a recent result of Alarc\'on and L\'opez. We also give a Forstneri\v c-Wold theorem for proper holomorphic embeddings (with respect to the given complex structure) of certain open Riemann surfaces into $(\C^*)^2$.
\end{abstract}

\maketitle
\tableofcontents

\section{Introduction}
\label{sec:introduction}

\noindent
It is a long-standing and difficult question of complex geometry whether every open (that is, non-compact) Riemann surface can be properly holomorphically embedded in 2-dimensional complex affine space $\C^2$. We refer the reader to the introductions of the papers \cite{ForstnericWold2009} and \cite{Ritter2013}, as well as the recent monograph \cite[\S\S8.9--8.10]{Forstneric2011}, for details of the progress made towards a solution of this problem. Of particular importance to this paper are the powerful techniques introduced by Wold in \cite{Wold2006a,Wold2006} where he proved that every finitely connected planar domain in $\C$ properly holomorphically embeds into $\C^2$, thereby extending a previous result of Globevnik and Stens\o nes \cite{GlobevnikStensones1995} for a restricted class of such domains. By applying Wold's techniques, Forstneri\v c and Wold \cite{ForstnericWold2009} subsequently proved that, given a holomorphic embedding of a compact bordered Riemann surface into $\C^2$ (see Definitions~\ref{def:borderedsurface} and \ref{def:embedding}), there exists a proper holomorphic embedding into $\C^2$ of the open Riemann surface that is its interior.

In \cite{Ritter2013} I adapted Wold's embedding techniques to the target manifold $\C\times\C^*$ and proved a corresponding Forstneri\v c-Wold embedding theorem for the interior of compact bordered Riemann surfaces, with added control over the homotopy class of the embedding. Using this, I obtained an \emph{Oka principle for proper holomorphic embeddings} of bounded finitely connected planar domains without isolated boundary points, namely that, given such a domain $X$, every continuous map $f : X \to \C\times\C^*$ can be deformed to a proper holomorphic embedding. Some progress towards generalising this to arbitrary finitely connected planar domains was made in \cite{LarussonRitter2012}. Here, the term \emph{Oka principle for proper holomorphic embeddings} comes by way of comparison to the Oka principle of Gromov \cite{Gromov1989}, which in this setting guarantees the existence of a holomorphic deformation of $f$, but without any claims as to its properness, injectivity, or immersivity.

By allowing deformations of the complex structure on an open Riemann surface we can investigate whether there are any topological obstructions to obtaining proper holomorphic embeddings into $\C^2$. Following this idea, \v Cerne and Forstneric \cite{CerneForstneric2002} proved that every open Riemann surface that is the interior of a compact bordered Riemann surface can be given a complex structure with respect to which it properly holomorphically embeds into $\C^2$. Extending this result, Alarc\'on and L\'opez \cite{AlarconLopez2013} recently applied the embedding techniques of Forstneri\v c and Wold to prove the same fact for arbitrary open Riemann surfaces, showing that the topology of the surface plays no role in the problem.

In the current paper we pursue a similar idea, considering proper holomorphic embeddings of arbitrary open Riemann surfaces $X$ on which we permit deformations of the complex structure, but with target $(\C^*)^2$ rather than $\C^2$ (this choice of target is explained later in the introduction following Corollary~\ref{cor:deformstructureembedding}). Rather than simply exhibiting the existence of such embeddings, however, our focus is on deforming a given continuous map $X \to (\C^*)^2$ into a proper holomorphic embedding that satisfies several additional conditions.

A simplified version of our main result is as follows. Note that at this point we do not yet consider any change to the complex structure on the open Riemann surface $X$.

\begin{theorem}
\label{thm:simplifiedmaintheorem}
Let $X$ be an open Riemann surface and $K \subset X$ be an $\OO(X)$-convex compact set. Let $f : X \to (\C^*)^2$ be a continuous map that restricts to a holomorphic embedding of an open neighbourhood of $K$. Then there exists a smooth isotopy of smooth embeddings $\Phi_t : X \to X$, $t \in [0,1]$, with $\Phi_0 = \id_X$, and a continuous family of continuous maps $f_t : \Phi_t(X) \to (\C^*)^2$, $t \in [0,1]$, with $f_0 = f$, such that $f_1 : \Phi_1(X) \to (\C^*)^2$ is a proper holomorphic embedding.

Furthermore, we can ensure that for all $t \in [0,1]$, $\Phi_t$ restricts to the identity on some fixed open neighbourhood of $K$, and that on some smaller open neighbourhood of $K$ the family $f_t$ restricts to a family of holomorphic embeddings that uniformly approximate $f$ on $K$.
\end{theorem}

We think of the isotopy of smooth embeddings $\Phi_t : X \to X$ as giving a smooth deformation of $X$ within itself to an open subset $\Phi_1(X)$ that can be properly holomorphically embedded into $(\C^*)^2$. It is then easy to see that the proper holomorphic embedding $f_1 : \Phi_1(X) \to (\C^*)^2$ is homotopic to $f\rvert_{\Phi_1(X)}$. Indeed, $f\rvert_{\Phi_1(X)} = f \circ \iota$, where $\iota : \Phi_1(X) \hookrightarrow X$ is the inclusion. The family $\Phi_t \circ \Phi_1^{-1} : \Phi_1(X) \to X$, $t \in [0,1]$, gives a homotopy from $\Phi_1^{-1}$ to $\iota$, so it suffices to show that $f \circ \Phi_1^{-1}$ is homotopic to $f_1$. This in turn follows by considering the family of maps $f_t \circ \Phi_t \circ \Phi_1^{-1} : \Phi_1(X) \to (\C^*)^2$, $t \in [0,1]$.

The proof of Theorem~\ref{thm:simplifiedmaintheorem} is given in Section~\ref{sec:maintheorem}. In fact, there we prove a stronger result (Theorem~\ref{thm:mainresult}) with interpolation on a discrete set $S \subset X$, provided that $f\rvert_S : S \to (\C^*)^2$ is a proper injection. In that case, we can also ensure that $\Phi_t$ restricts to the identity on a fixed neighbourhood of $S$, and that $f_t\rvert_S = f\rvert_S$ for all $t \in [0,1]$.

At this point we rephrase the conclusions of Theorems~\ref{thm:simplifiedmaintheorem} and \ref{thm:mainresult} by permitting deformations to the complex structure on $X$. Using $\Phi_t$ to pull back the complex structure from $\Phi_t(X)$ to $X$, we obtain the following corollary of Theorem~\ref{thm:mainresult}.

\begin{corollary}
\label{cor:maincorollary}
Let $X$ be an open Riemann surface with complex structure operator $J$. Let $K \subset X$ be an $\OO(X)$-convex compact set, and $S \subset X$ be a discrete set. Let $f : X \to (\C^*)^2$ be a continuous map that restricts to a holomorphic embedding of a neighbourhood of $K$ and such that $f\rvert_S : S \to (\C^*)^2$ is a proper injection. Then there exists a smooth deformation of $J$ through complex structures to a complex structure $\tilde J$ on $X$, and a continuous deformation of $f$ to a map $\tilde f : X \to (\C^*)^2$ that is a proper holomorphic embedding with respect to $\tilde J$, such that $\tilde f$ approximates $f$ uniformly on $K$ and $\tilde f\rvert_S = f\rvert_S$.

Furthermore, we can ensure that $J$ does not change on a fixed neighbourhood of $K \cup S$, and that the deformation of $f$ is through maps that restrict to holomorphic embeddings of a fixed smaller neighbourhood of $K$, that uniformly approximate $f$ on $K$ and equal $f$ on $S$.
\end{corollary}
 
We remark that, since every open Riemann surface is Stein, the deformation of $J$ to $\tilde J$ is trivially through Stein complex structures on $X$.

Given a complex manifold $Y$, one can ask whether it possesses the following so-called \emph{basic Oka property}: Can every continuous map $f: X \to Y$ from a Stein manifold $X$ into $Y$ be deformed into a holomorphic map $\tilde f : X \to Y$? We can also ask a stronger question as follows. Let $K \subset X$ be a compact $\OO(X)$-convex set and $A \subset X$ be a closed complex subvariety, and suppose that the continuous map $f : X \to Y$ is holomorphic on a neighbourhood of $K$, and has holomorphic restriction to $A$. Can $f$ be deformed to a holomorphic map $\tilde f$ while keeping it fixed on $A$ and almost fixed on $K$? This is known as the \emph{basic Oka property with approximation and interpolation}, and manifolds $Y$ that satisfy this property are said to be \emph{Oka}. We refer the reader to the surveys \cite{Forstneric2013,ForstnericLarusson2011} and the monograph \cite[Chapter~5]{Forstneric2011} for excellent treatments of the Oka principle in complex geometry.

Noting that Oka manifolds are relatively rare, Forstneri\v c and Slapar \cite{ForstnericSlapar2007,ForstnericSlapar2007a} consider the problem of deforming a continuous map $f : X \to Y$, where $X$ is a Stein manifold and $Y$ is a complex manifold, to a map $\tilde f$ holomorphic with respect to some deformed complex structure on $X$. They prove that this so-called \emph{soft Oka property} holds for arbitrary complex manifold targets $Y$. Additionally, if $f$ is already holomorphic on a complex subvariety $A \subset X$ then it can be kept fixed there, so that $\tilde f\rvert_A = f\rvert_A$, and the complex structure on $X$ can also be left unchanged in a neighbourhood of $A$. Our Corollary~\ref{cor:maincorollary} is therefore very similar in nature to the results of Forstneri\v c and Slapar, except that by restricting our attention to a special target manifold we can ensure our final maps are not just holomorphic, but are proper holomorphic embeddings that both approximate and interpolate the given continuous function on suitable sets. For that reason we call Corollary~\ref{cor:maincorollary} a \emph{soft Oka principle for proper holomorphic embeddings of open Riemann surfaces into $(\C^*)^2$}.

If we let $K = \varnothing$ and $S = \varnothing$ in Corollary~\ref{cor:maincorollary}, we obtain the following statement on embeddings of open Riemann surfaces into $(\C^*)^2$.

\begin{corollary}
\label{cor:deformstructureembedding}
Let $X$ be an open Riemann surface and $f : X \to (\C^*)^2$ be a continuous map. Then, after possibly deforming the complex structure on $X$, there exists a proper holomorphic embedding $\tilde f : X \to (\C^*)^2$ that is homotopic to $f$.
\end{corollary}

Thus there are no topological obstructions, in any homotopy class, to the existence of proper holomorphic embeddings of open Riemann surfaces into $(\C^*)^2$.

Taking $f : X \to (\C^*)^2$ to be a constant map, we see that \emph{every open Riemann surface has a complex structure such that it properly holomorphically embeds into $(\C^*)^2$}. Since every non-compact smooth orientable surface can be given a complex structure, we also see that every such smooth surface has a complex structure with respect to which it can be properly holomorphically embedded into $(\C^*)^2$.

Finally, note that without any significant changes in the proofs, all of the results in the present paper continue to hold with the target manifold $(\C^*)^2$ replaced by $\C^2$ or $\C\times\C^*$. There are two main reasons, then, why we choose to work with target $(\C^*)^2$ rather than $\C^2$ or $\C\times\C^*$. For one, $(\C^*)^2$ is homotopy equivalent to $S^1 \times S^1$, and is therefore non-contractible, in contrast to $\C^2$. Thus, while every map $X \to \C^2$ is null-homotopic, so that any two such maps are trivially homotopic, this is no longer the case for maps $X \to (\C^*)^2$, and it becomes of interest to consider the homotopy class of such maps. Second, while $\C^2$ and $\C\times\C^*$ possess the so-called \emph{density property}, an essential ingredient in Wold's embedding techniques, it is not known whether $(\C^*)^2$ has this property. We must therefore make use of the related, but different, \emph{volume density property}, which $(\C^*)^2$ is known to have. We believe this is the first time that Forstneri\v c and Wold's embedding techniques have been adapted to such a target.

The content of the remainder of the paper is as follows.

In Section~\ref{sec:embeddingtechniques} we show that the techniques developed by Forstneri\v c and Wold for embedding certain open Riemann surfaces into $\C^2$ can be adapted to the target $(\C^*)^2$. These results become our main tools in Section~\ref{sec:enlargingdomainofembedding}. However, from the results in Section~\ref{sec:embeddingtechniques} we also obtain, essentially for free, Theorem~\ref{thm:embeddingcompactbordered} on properly holomorphically embedding the interior of a compact bordered Riemann surface into $(\C^*)^2$. This result mirrors that previously obtained by Forstneri\v c and Wold for target $\C^2$ in \cite{ForstnericWold2009}, and by the author for target $\C\times\C^*$ in \cite{Ritter2013}. It is of interest to note that this result implies Corollary~\ref{cor:deformstructureembedding} holds \emph{without} any deformation of the complex structure on $X$, when $X$ is the interior of a compact bordered Riemann surface that holomorphically embeds into $(\C^*)^2$.

Section~\ref{sec:enlargingdomainofembedding} contains Lemma~\ref{lem:enlargingdomainofembedding}, the main technical result required in the proof of Theorem~\ref{thm:mainresult}. Using the tools developed in Section~\ref{sec:embeddingtechniques} we explain how to deform a holomorphic embedding of a neighbourhood of a compact bordered Riemann surface $L \subset X$, where $X$ is an open Riemann surface, so as to obtain a holomorphic embedding of a larger compact bordered Riemann surface $L'$, while maintaining necessary control on the image of $\overline{L' \setminus L}$.

In the final Section~\ref{sec:maintheorem}, we first give Lemma~\ref{lem:deformationunion} demonstrating how to construct a family of smooth embeddings $\Phi_t : X \to X$ of an open Riemann surface $X$ into itself, using a sequence of diffeotopies of $X$ that satisfies certain conditions on a compact exhaustion of $X$. Then, using Lemma~\ref{lem:enlargingdomainofembedding} and Lemma~\ref{lem:deformationunion}, together with basic techniques from Morse theory, we prove the main result of the paper, Theorem~\ref{thm:mainresult}.

I wish to thank Erlend Wold and Erik L\o w for helpful discussions during the preparation of this paper. I also thank Finnur L\'arusson for his constructive comments on an earlier draft.

\section{Embedding techniques for $(\C^*)^2$}
\label{sec:embeddingtechniques}

\noindent
In this section we show that Forstneri\v c and Wold's techniques \cite{ForstnericWold2009,Wold2006a,Wold2006} for embedding certain open Riemann surfaces into $\C^2$ can be adapted to the target $(\C^*)^2$. These techniques were previously shown to adapt to the target $\C\times\C^*$ in the papers \cite{LarussonRitter2012,Ritter2013}. Apart from some differences in geometry due to working in $(\C^*)^2$, the essential difference here is that $(\C^*)^2$ is only known to have the \emph{volume density property}, whereas both $\C^2$ and $\C\times\C^*$ have the \emph{density property} (see Definition~\ref{def:volumedensityproperty}). This necessitates some technical changes in our arguments, compared to previous treatments. We also prove slightly stronger versions of corresponding results in the above references, as our main result requires continuous 1-parameter families of maps fixed on a finite set.

We begin by fixing some notation. Equip the Stein manifold $(\C^*)^2$ with the Riemannian distance function $d$ that comes from the inclusion $(\C^*)^2 \subset \C^2$. Let $\pi_1, \pi_2 : (\C^*)^2 \to \C^*$ denote the projection onto the first and second components of $(\C^*)^2 = \C^* \times \C^*$, respectively. Given $r > 0$, denote by $A_r$ the annulus
\[ A_r = \{z \in \C : 1/(r+1) < \lvert z \rvert < r+1 \} \subset \C^*\,,\]
and let $P_r$ be the set
\[ P_r = A_r \times A_r \subset (\C^*)^2\,.\]
For $r \le 0$, let $A_r = \varnothing$ and $P_r = \varnothing$.
Taking a strictly increasing sequence of positive real numbers $r_j \to \infty$, we obtain a compact exhaustion $(\overline{P}_{r_j})_{j \in \N}$ of $(\C^*)^2$ by $\OO((\C^*)^2)$-convex sets. (Recall that a compact set $K$ in a complex manifold $Z$ is said to be \emph{$\OO(Z)$-convex} if, for every $x \in Z \setminus K$, there exists $f \in \OO(Z)$ such that $\lvert f(x) \rvert > \lVert f \rVert_K$.)

\begin{definition}
\label{def:borderedsurface}
A \emph{bordered Riemann surface} $L$ is a (not necessarily compact) 2-dimen\-sional smooth manifold with (possibly empty) boundary, equipped with a complex structure on its interior that is compatible with the given smooth structure. We denote the boundary of $L$ by $bL$ and the interior by $\mathring{L} = L \setminus bL$. 
\end{definition}

In this paper we often consider compact bordered Riemann surfaces, for which $bL$ consists of a finite number of components, each diffeomorphic to a circle. In this section we also consider certain non-compact bordered Riemann surfaces, such as those resulting from the removal of a single point from each boundary component of a compact bordered Riemann surface (in which case $bL$ has finitely many components, each diffeomorphic to $\R$).

\begin{definition}
\label{def:embedding}
An \emph{embedding} $L \to (\C^*)^2$ of a bordered Riemann surface $L$ is a smooth injective immersion that is a homeomorphism onto its image, and is holomorphic on $\mathring{L}$. We will mention explicitly when an embedding has the additional property of being proper.
\end{definition}

In later sections it will often be the case that a compact bordered Riemann surface $L$ is contained in an open Riemann surface $X$, and that an embedding $L \to (\C^*)^2$ is the restriction to $L$ of a holomorphic embedding of an open neighbourhood $U \subset X$ containing $L$, but this will not be assumed here.

\begin{definition}
Let $X$ and $Y$ be smooth manifolds, and let $\mathcal{F}$ be a set of functions from $X$ to $Y$. Let $k \in \N \cup \{\infty\}$. We say a map $f(t,x) : [0,1] \times X \to Y$ is a \emph{$\CC^k$-isotopy of $\mathcal{F}$-maps} if $f$ is of class $\CC^k$ in $(t,x) \in [0,1] \times X$ and, for each $t\in [0,1]$, we have $f(t,\cdot) \in \mathcal{F}$. We often write $f_t(x)$ for $f(t,x)$.
\end{definition}

Unless mentioned otherwise, the parameter space for homotopies and isotopies is assumed to be $I = [0,1]$.

The notion of an \emph{exposed point} of a bordered Riemann surface embedded in $\C^2$ was introduced in \cite[Definition~4.1]{ForstnericWold2009}, and given in \cite[Definition~7]{Ritter2013} for embeddings into $\C\times\C^*$. The precise notion we require here is as follows.

\begin{definition}
\label{def:exposedpoint}
Let $f : L \to (\C^*)^2$ be an embedding of the bordered Riemann surface $L$ into $(\C^*)^2$, and let $p \in L$. We say that $p$ is an \emph{$f$-exposed point} (or that $f(p)$ is an \emph{exposed point}) if the complex line
\[
  \pi_1^{-1}(\pi_1(f(p))) = \{\pi_1(f(p))\}\times \C^* \subset (\C^*)^2
\]
intersects $f(L)$ only at $f(p)$, and the intersection is transverse.
\end{definition}

Note that this differs slightly from the definition found in \cite{Ritter2013}, where $\pi_2$ was used to project onto the $\C^*$ component of $\C\times\C^*$.

The following result on exposing boundary points was proved in \cite[Theorem~4.2]{ForstnericWold2009} for the target $\C^2$, where it was used to modify an embedding of a compact bordered Riemann surface in $\C^2$ so that Wold's embedding techniques could then be applied. In \cite[Theorem~5]{Ritter2013} it was shown that the result also holds for embeddings into $\C\times\C^*$, and that in that case the initial and final maps are homotopic. Here, we need the result for embeddings into $(\C^*)^2$, together with some small additional requirements.

\begin{proposition}
\label{prop:exposeboundarypoints}
Let $L$ be a compact bordered Riemann surface. For each component $\gamma_j$ of $bL$, $j = 1, \dots, m$, let $a_j \in \gamma_j$ and let $U_j$ be an open neighbourhood of $a_j$ in $L$. Let $b_1, \dots, b_n \in L \setminus \bigcup\limits_{j=1}^m \overline{U}_j$. Given an embedding $f : L \to (\C^*)^2$ and $\epsilon > 0$, there exists a homotopy of maps $f_t : L \to (\C^*)^2$ satisfying the following conditions.
\begin{itemize}
 \item $f_0 = f$.
 \item $f_t(b_k) = f(b_k)$ for all $t \in I$ and all $k = 1, \dots, n$.
 \item $\sup\limits_{x \in L \setminus \cup_{j=1}^m U_j}d(f_t(x),f(x)) < \epsilon$ for all $t \in I$.
 \item For all $t \in I$, $f_t$ restricts to an embedding of $L \setminus \bigcup\limits_{j=1}^m \overline{U}_j$.
 \item $f_1 : L \to (\C^*)^2$ is an embedding such that all the points $a_1, \dots, a_m$ are $f_1$-exposed.
\end{itemize}
Furthermore, if $r > 0$ is such that $f(\overline{U}_j) \subset (\C^*)^2 \setminus \overline{P}_r$ for $j = 1,\dots,m$, then we can ensure that $f_t(\overline{U}_j) \subset (\C^*)^2 \setminus \overline{P}_r$ for all $t \in I$ and all $j = 1,\dots,m$. Additionally, we can ensure that $\pi_1(f_1(a_j)) \notin \overline{A}_r$ for $j = 1, \dots, m$.
\end{proposition}

\begin{proof}


Directly following \cite{ForstnericWold2009} and \cite{Ritter2013}, we approximate $f$ in the $\CC^1$ topology on $L \setminus \bigcup\limits_{j=1}^m U_j$ by an embedding $F : L \to (\C^*)^2$ such that $F(b_k) = f(b_k)$ for $k = 1,\dots,n$ and such that the points $a_j$, $j = 1,\dots,m$, are all $F$-exposed. The argument given in \cite{Ritter2013} furnishes a homotopy $f_t : L \to (\C^*)^2$ linking $f$ to $F$, which on $L \setminus \bigcup\limits_{j=1}^m \overline{U}_j$ is given by linear interpolation. Thus $f_t(b_k) = f(b_k)$ for all $t \in I$ and $k = 1,\dots,n$, and provided the approximation of $f$ by $F$ was sufficiently close, all intermediate maps restrict to embeddings of $L \setminus \bigcup\limits_{j=1}^m \overline{U}_j$.

Examining the proofs in \cite{ForstnericWold2009,Ritter2013}, it is clear that, given any $r > 0$, we are free to choose the points $f_1(a_j)$ so that they satsify $\pi_1(f_1(a_j)) \notin \overline{A}_r$ for $j = 1, \dots, m$. If $f(\overline{U}_j) \subset (\C^*)^2 \setminus \overline{P}_r$ for some $r > 0$ and all $j = 1,\dots,m$, then it is clear that we may also choose the curves $\lambda_j$ required in the proof that link each $f(a_j)$ to $f_1(a_j)$ to lie entirely outside of $\overline{P}_r$. For each $j$, the image $f_t(\overline{U}_j)$, $t \in I$, can be constrained to lie in an arbitrarily small neighbourhood of the set $\lambda_j \cup \overline{U}_j$, ensuring that $f_t(\overline{U}_j) \subset (\C^*)^2 \setminus \overline{P}_r$ for all $t \in I$ and all $j = 1,\dots,m$.
\end{proof}

\begin{remark}
\label{rem:exposeboundarypoints}
Suppose in Proposition \ref{prop:exposeboundarypoints} that $K \subset \mathring{L}$ is a compact set such that $f(L \setminus \mathring{K}) \subset (\C^*)^2 \setminus \overline{P}_r$ for some $r > 0$. By the proposition we can ensure that $f_t$ uniformly approximates $f$ sufficiently closely on $L \setminus \bigcup\limits_{j=1}^m U_j$ so that, together with the control over $f_t(\overline{U}_j)$, we have $f_t(L \setminus \mathring{K}) \subset (\C^*)^2 \setminus \overline{P}_r$ for all $t \in I$. This will be essential in the proof of Lemma \ref{lem:enlargingdomainofembedding}.
\end{remark}

The \emph{$\C$-nice projection property} concerns properties of a collection unbounded curves embedded in $\C^2$ or $\C\times\C^*$ after being projected onto $\C$, and was first stated explicitly in \cite[Definition~2.1]{Kutzschebauch2009}, following its implicit use in \cite[Lemma~2.1]{Wold2006a} and \cite[Lemma~1]{Wold2006}. An additional, necessary, condition was added to the definition in \cite[Definition~1]{Ritter2013}. We require a variant called the \emph{$\C^*$-nice projection property}, given in \cite[Definition~2]{LarussonRitter2012}, in which curves embedded in $\C\times\C^*$ or $(\C^*)^2$ are projected onto $\C^*$. For the convenience of the reader, we include the definition from \cite{LarussonRitter2012}, stated for curves in $(\C^*)^2$.

\begin{definition}
\label{def:niceprojectionproperty}
 Let $\gamma_1, \dots, \gamma_m : \R \to (\C^*)^2$ be pairwise disjoint, smoothly embedded curves in $(\C^*)^2$. For $j = 1, \dots, m$, let $\Gamma_j \subset (\C^*)^2$ be the image of $\gamma_j$, and set $\Gamma = \bigcup\limits_{j=1}^m \Gamma_j$. We say that the collection $\gamma_1,\dots,\gamma_m$ has the \emph{$\C^*$-nice projection property} if there is a holomorphic automorphism $\alpha \in \Aut((\C^*)^2)$ such that, if $\beta_j = \alpha \circ \gamma_j$ and $\Gamma' = \alpha(\Gamma)$, the following conditions hold.
\begin{enumerate}
 \item \label{def:niceprojectionproperty:proper} For every compact set $K \subset \C^*$ there exists $s > 0$ such that $\pi_2(\beta_j(t)) \notin K$ for all $\lvert t \rvert > s$ and all $j = 1,\dots,m$.
 \item There exists $M > 0$ such that for all $r \ge M$:
\begin{enumerate}
 \item $\C^* \setminus (\pi_2(\Gamma')\cup\overline{A}_r)$ does not contain any relatively compact connected components.
 \item $\pi_2$ is injective on $\Gamma' \setminus \pi_2^{-1}(A_r)$.
\end{enumerate}
\end{enumerate}
\end{definition}

Note that condition (\ref{def:niceprojectionproperty:proper}) above states that each map $\pi_2\circ\beta_j$ is proper into $\C^*$, and thus each curve $\beta_j$ a proper embedding of $\R$ into $(\C^*)^2$. The definition is clearly independent of the particular parameterisations, so we may also refer to the set $\Gamma$ as having the $\C^*$-nice projection property.

\begin{proposition}
\label{prop:niceprojectionproperty}
Let $L$, $\gamma_j$, $a_j$, $U_j$, and $b_k$ be as in Proposition \ref{prop:exposeboundarypoints}. Let $f : L \to (\C^*)^2$ be an embedding such that each point $a_j \in \gamma_j$ is $f$-exposed, $j = 1, \dots, m$. Given $\epsilon > 0$, there exists a homotopy of embeddings $f_t : L \setminus \{a_1,\dots,a_m\} \to (\C^*)^2$ satisfying the following conditions.
\begin{itemize}
 \item $f_0 = f\rvert_{L\setminus \{a_1,\dots,a_m\}}$.
 \item $\pi_1 \circ f_t = \pi_1 \circ f\rvert_{L\setminus \{a_1,\dots,a_m\}}$ for all $t \in I$.
 \item $f_t(b_k) = f(b_k)$ for all $t \in I$ and all $k = 1, \dots, n$.
 \item $\sup\limits_{x \in L \setminus \cup_{j=1}^m U_j}d(f_t(x),f(x)) < \epsilon$ for all $t \in I$.
 \item $f_1(bL \setminus \{a_1,\dots,a_m\})$ possesses the $\C^*$-nice projection property.
\end{itemize}

\end{proposition}

\begin{proof}
Because each $a_j$ is an $f$-exposed point in the boundary of the compact bordered Riemann surface $L$, it follows that every neighbourhood of $\pi_1(f(a_j))$ contains a point $\alpha_j \in \C^* \setminus \pi_1(f(L))$. For each $j = 1,\dots,m$, denote the compact line segment from $\pi_1(f(a_j))$ to $\alpha_j$ by $J_j \subset \C^*$. We may then assume that $J_j \setminus \{\pi_1(f(a_j))\} \subset \C^* \setminus \pi_1(f(L))$ for each $j$. Furthermore, we can choose the argument of $\alpha_j - \pi_1(f(a_j))$ to be any value in the interval $(\theta_j - \pi/2, \theta_j + \pi/2)$, where $\theta_j$ is the argument of the outward normal to $\pi_1(f(\gamma_j))$ at $\pi_1(f(a_j))$.

Define an isotopy of biholomorphic maps $\tilde{g}_t$ on $(\C^*)^2 \setminus (\{J_1,\dots,J_m\} \times \C^*)$ starting at the identity by
\[ \tilde{g}_t(z,w) = \left(z, w\cdot\prod_{j=1}^{m}\left(1 - t\frac{\alpha_j-\pi_1(f(a_j))}{z-\pi_1(f(a_j))}\right)\right),\quad t\in I\,.\]

Let $\tilde{f}_t = \tilde{g}_t \circ f\rvert_{L \setminus \{a_1,\dots,a_m\}}$. Each $\tilde{f}_t$ is then an embedding of $L \setminus \{a_1,\dots,a_m\}$ into $(\C^*)^2$ and for $t > 0$ the curves $\tilde{f}_t(\gamma_j \setminus \{a_j\})$ all have unbounded second component. By choosing each $\alpha_j$ sufficiently close to $\pi_1(f(a_j))$ we can ensure that the family $\tilde{f}_t$ uniformly approximates $f$ on $L \setminus \bigcup\limits_{j=1}^m U_j$. A generic choice of argument for each $\alpha_j - \pi_1(f(a_j))$ then ensures the $\C^*$-nice projection property is satisfied by the collection of curves $\tilde{f}_1 (\gamma_j \setminus \{a_j\})$, $j = 1,\dots,m$.

It remains to explain how to adjust the values at the points $b_k$ without destroying the properties already established above. To do this, we solve $e^{h_t(z)}\prod\limits_{j=1}^{m}\left(1 - t\frac{\alpha_j-\pi_1(f(a_j))}{z-\pi_1(f(a_j))}\right) = 1$ at the points $z = \pi_1(f(b_1)),\dots,\pi_1(f(b_n))$ for some suitable continuous family $h_t \in \OO(\C)$. First note that for $k = 1,\dots,n$ there exist values $h_t(\pi_1(f(b_k)))$, all close to $0$, that solve the above equation and depend continuously on $t \in I$. Let $h_t \in \OO(\C)$ be a continuous family of polynomials interpolating these values for each $t \in I$, with $h_0 = 0$. By choosing the $\alpha_j$ sufficiently close to $\pi_1(f(a_j))$, the required values $h_t(\pi_1(f(b_k)))$ are in turn so close to $0$ that $h_t$ can be made arbitrarily small, uniformly in $t$, on a large disc containing $\pi_1(f(L)) \subset \C^* \subset \C$. Defining
\[ g_t(z,w) = \left(z, w\cdot e^{h_t(z)}\prod_{j=1}^{m}\left(1 - t\frac{\alpha_j-\pi_1(f(a_j))}{z-\pi_1(f(a_j))}\right)\right),\quad t\in I\,.\]
and $f_t = g_t \circ f$, we obtain the desired family of maps.
\end{proof}

\begin{remark}
\label{rem:niceprojectionproperty}
A similar statement to Remark \ref{rem:exposeboundarypoints} also applies to Proposition \ref{prop:niceprojectionproperty}. That is, suppose that $K \subset \mathring{L}$ is a compact set such that $f(L \setminus \mathring{K}) \subset (\C^*)^2 \setminus \overline{P}_r$ for some $r > 0$. If we also suppose that $\pi_1(f(\overline{U}_j)) \subset \C^*\setminus \overline{A}_r$ for $j = 1,\dots,m$, then the fact that the family $f_t$ uniformly approximates $f$ on $L \setminus \bigcup\limits_{j=1}^m U_j$, together with the fact that $\pi_1(f_t(\overline{U}_j \setminus \{a_j\})) = \pi_1(f(\overline{U}_j \setminus \{a_j\}))$ for all $t \in I$, ensures that we have $f_t(L \setminus (\{a_1,\dots,a_m\}\cup\mathring{K})) \subset (\C^*)^2 \setminus \overline{P}_r$ for all $t \in I$.
\end{remark}

Let $Z$ be a complex manifold equipped with a holomorphic volume form $\omega$ (that is, a non-vanishing holomorphic differential form of degree $\dim Z$), and let $D \subset Z$ be an open set. A holomorphic map $\phi : D \to Z$ is said to be \emph{volume preserving} if it satisfies $\phi^*\omega = \omega$ on $D$. Given a holomorphic vector field $V$ on $Z$, its \emph{divergence} with respect to $\omega$ is the holomorphic function $\div_\omega V$ on $Z$ that satisfies
\[ \LL_V\omega = \div_\omega V \cdot \omega\,, \]
where $\LL_V\omega$ is the Lie derivative of $\omega$ along $V$. If we let $\phi_t$ denote the flow of $V$, $t \in \C$, it is easy to see that $\phi_t$ is volume preserving on $Z$ if and only if $\div_\omega V = 0$, in which case we say that $V$ is \emph{divergence free}. The divergence free holomorphic vector fields form a Lie subalgebra of the Lie algebra of all holomorphic vector fields on $Z$.

Recall that a vector field $V$ on $Z$ is said to be \emph{complete} if its flow $\phi_t(x)$ exists for all $t \in \C$, for all $x \in Z$. Note that if $\phi_t(x)$ is the flow of a complete (divergence free) holomorphic vector field on $Z$ then for every $t \in \C$, $\phi_t(\cdot) : Z \to Z$ is a (volume preserving) holomorphic automorphism of $Z$.

Based on the preceding observations, the following notion was introduced by Varolin in \cite{Varolin2000,Varolin2001}, generalising a property previously shown by Anders\'en \cite{Andersen1990} to hold for $\C^n$, $n \ge 1$, with holomorphic volume form $\omega = dz_1 \wedge \cdots \wedge dz_n$.

\begin{definition}
\label{def:volumedensityproperty}
Let $Z$ be a complex manifold with holomorphic volume form $\omega$. We say $Z$ has the \emph{volume density property} if the Lie algebra generated by the complete divergence free holomorphic vector fields on $Z$ is dense in the Lie algebra of all divergence free holomorphic vector fields on $Z$, in the compact open topology.
\end{definition}

A closely related notion is that of the \emph{density property}, which parallels the above definition except that we omit the words `divergence free'. Anders\'en and Lempert proved in \cite{AndersenLempert1992} that $\C^n$ has the density property for all $n \ge 2$, and consequently showed that for $n \ge 2$ the group $\Aut(\C^n)$ of holomorphic automorphisms of $\C^n$ is sufficiently large that certain injective holomorphic maps of starshaped domains in $\C^n$ can be approximated by automorphisms of $\C^n$. They also proved a corresponding result on the approximation of injective volume preserving holomorphic maps by volume preserving automorphisms of $\C^n$. These results, known collectively as the \emph{Anders\'en-Lempert theorem}, were further developed by Forstneri\v c and Rosay \cite{ForstnericRosay1993,ForstnericRosay1994}, who proved the approximation of isotopies of injective (volume preserving) holomorphic maps on more general subsets of $\C^n$ by isotopies of (volume preserving) automorphisms of $\C^n$. In \cite{Varolin2000} Varolin observed that the same arguments give the Anders\'en-Lempert theorem on any Stein manifold with the (volume) density property.


The fact that  $(\C^*)^2$ has the volume density property with respect to the volume form $\omega = (zw)^{-1} dz \wedge dw$ was established by Varolin in \cite{Varolin2001}. It is still not known whether $(\C^*)^2$ has the density property. We therefore require the following version of the Anders\'en-Lempert theorem for Stein manifolds with the volume density property. A detailed proof of the corresponding result for Stein manifolds with the density property was given in \cite[Appendix]{Ritter2013}; that proof adapts immediately to the volume preserving situation in a standard manner, as explained in \cite{ForstnericRosay1994} (see also Remark (\ref{rem:andersenlempert}).

\begin{theorem}
\label{thm:andersenlempert}
Let $Z$ be an $n$-dimensional Stein manifold with the volume density property. Let $\Omega \subset Z$ be an open set satisfying $H^{n-1}(\Omega; \C) = 0$, and let $\phi_t : \Omega \to Z$ be a $\CC^1$-isotopy of volume preserving injective holomorphic maps such that $\phi_0$ is the inclusion $\Omega \hookrightarrow Z$. Suppose $K \subset \Omega$ is a compact set such that $\phi_t(K)$ is $\OO(Z)$-convex for all $t \in I$. Then, given $\epsilon > 0$, there exists a continuous family $\sigma_t$, $t \in I$, of volume preserving holomorphic automorphisms of $Z$ such that $\sup_{x \in K}d(\sigma_t(x),\phi_t(x)) < \epsilon$ for all $t \in I$, where $d$ is any Riemannian distance function on $Z$.
\end{theorem}

\begin{remark}
\label{rem:andersenlempert}
The cohomological assumption $H^{n-1}(\Omega; \C) = 0$ in the statement of Theorem \ref{thm:andersenlempert} is sufficient, but not necessary. As explained in \cite{ForstnericRosay1994}, the isotopy $\phi_t$ is the flow of a time-dependent divergence free vector field $X_t$ defined on $\Omega_t = \phi_t(\Omega)$, to which we may associate a closed holomorphic $(n-1)$-form $\alpha_t$ on $\Omega_t$ using the pairing induced by $\omega$. The cohomological assumption then implies that $\alpha_t$ is exact, and the argument proceeds. However, if the isotopy $\phi_t$ restricts to the inclusion on some connected component $\Omega'$ of $\Omega$ then it suffices to take $\alpha_t = 0$ on $\Omega'$, which is trivially exact. Thus we only require $H^{n-1}(\tilde{\Omega}; \C) = 0$ on those components $\tilde{\Omega}$ of $\Omega$ where $\phi_t$ is not the inclusion.
\end{remark}

The following technical lemma is the main ingredient in Wold's method for embedding Riemann surfaces, and was first proved by Wold \cite[Lemma~1]{Wold2006} for a collection of curves in $\C^2$ with the $\C$-nice projection property. In \cite[Lemma~4]{Ritter2013} the proof was adapted to work for curves in $\C\times\C^*$ with the $\C$-nice projection property, and in \cite[Lemma~3]{LarussonRitter2012} for curves with the $\C^*$-nice projection property (still in $\C\times\C^*$). Here, we state the result for curves in $(\C^*)^2$ with an additional condition on fixing finitely many points, and explain the necessary changes in the proof.

\begin{lemma}
\label{lem:woldlemma}
Let $K \subset (\C^*)^2$ be an $\OO((\C^*)^2)$-convex compact set and let $\gamma_1, \dots, \gamma_m$ be pairwise disjoint, smoothly embedded curves in $(\C^*)^2$ satisfying the $\C^*$-nice projection property (Definition \ref{def:niceprojectionproperty}). Let $\Gamma_j$ be the image of $\gamma_j$, $j = 1,\dots,m$, and set $\Gamma = \bigcup\limits_{j=1}^m \Gamma_j$. Suppose that $K \cap \Gamma = \varnothing$, and let $q_1, \dots, q_n \in (\C^*)^2 \setminus \Gamma$. Then, given $R > 0$ and $\epsilon > 0$, there exists a continuous family of automorphisms $\theta_t \in \Aut((\C^*)^2)$, $t \in I$, satisfying the following conditions.
\begin{itemize}
\item $\theta_0 = \id_{(\C^*)^2}$.
\item $\sup_{x \in K}d(\theta_t(x), x) < \epsilon$ for all $t \in I$.
\item $\theta_t(q_k) = q_k$ for $k = 1, \dots, n$, for all $t \in I$.
\item $\theta_1(\Gamma) \subset (\C^*)^2 \setminus \overline{P}_R$.
\end{itemize}
\end{lemma}

Before proving Lemma \ref{lem:woldlemma}, we give the following result that will be required in the proof.

\begin{lemma}
\label{lem:shrinkcurve}
Let $\gamma : [0,1] \to (\C^*)^2$ be a smoothly embedded compact curve with image $\Gamma = \gamma([0,1])$. Equip $(\C^*)^2$ with the standard volume form $\omega = (zw)^{-1} dz \wedge dw$. Let $p \in \Gamma$, and fix open neighbourhoods $V$ of $\Gamma$ and $W$ of $p$. Then there exists an open neighbourhood $U \subset V$ of $\Gamma$ and a $\CC^1$-isotopy of injective holomorphic maps $\chi_t : U \to V$ satisfying the following properties.
\begin{itemize}
 \item $\chi_0$ is the inclusion $U \hookrightarrow V$.
 \item $\chi_1(\Gamma) \subset W$.
 \item $\chi_t^*\omega = \omega$ for all $t \in I$.
\end{itemize}
\end{lemma}
\begin{proof}

Let $p' = \gamma^{-1}(p) \in [0,1]$. Using a $t$-linear contraction that keeps $p'$ fixed, define a $\CC^\infty$-isotopy of smooth embeddings $\psi'_t : [0,1] \to [0,1]$ that starts at the identity map and satisfies $\psi'_1([0,1]) \subset \gamma^{-1}(W)$. Then $\psi_t = \gamma \circ \psi'_t \circ \gamma^{-1} : \Gamma \to \Gamma$ is a $\CC^\infty$-isotopy of smooth embeddings, starting at the identity, such that $\psi_1(\Gamma) \subset W$.

We now wish to apply Theorem~1.7 from \cite{ForstnericLowOvrelid2001} (see also the remark following that theorem regarding totally real submanifolds with boundary). Thinking of $\psi_t$ as a family of maps into $(\C^*)^2$, and then noting that $\psi_t^*\omega = 0$ for all $t$, we see that $\psi_t$ is a totally real $\omega$-flow of class $\CC^\infty$ in the terminology of \cite{ForstnericLowOvrelid2001}. Although the result \cite[Theorem~1.7]{ForstnericLowOvrelid2001} is given for $\omega$ the standard volume form on $\C^n$, its proof only requires that $\omega$ be closed (see also the remark preceding the statement of Theorem~1.7 in \cite{ForstnericLowOvrelid2001}) and that $\omega$ induces via contraction a non-degenerate pairing between holomorphic $(n-1)$-forms and holomorphic vector fields. The proof thus also applies in the current situation. We may therefore approximate $\psi_t$ on $\Gamma$, uniformly in $t$, by a $\CC^1$-isotopy of injective holomorphic maps $\chi_t : U \to (\C^*)^2$ defined on a small neighbourhood $U$ of $\Gamma$, such that $\chi_0$ is the inclusion and $\chi_t^*\omega = \omega$ for all $t$. Note that as the approximation improves, the size of the neighbourhood $U$ in general decreases. Provided the approximation is sufficiently close, and after possibly shrinking $U$ about $\Gamma$, the family $\chi_t$ satisfies all conclusions of the lemma.
\end{proof}

\begin{proof}[Proof of Lemma \ref{lem:woldlemma}]
We follow the proofs given in \cite[Lemma~3]{LarussonRitter2012} and \cite[Lemma~4]{Ritter2013} for curves in the complex manifold $\C\times\C^*$, which has the density property, explaining the required changes due to our ambient manifold being $(\C^*)^2$ and having instead the volume density property.

As in the cited references, we may assume that the automorphism in Definition~\ref{def:niceprojectionproperty} has already been applied, so that the conditions of the $\C^*$-nice projection property hold directly for the curves $\gamma_1,\dots,\gamma_m$. Since the union of an $\OO((\C^*)^2)$-convex set with finitely many points is still $\OO((\C^*)^2)$-convex, we may assume that $q_1,\dots,q_n \in K$. We show how to obtain all the conditions except for fixing the points $q_1,\dots,q_n$, which is a standard addition that we explain at the end.

Let $K'$ be a slightly larger $\OO((\C^*)^2)$-convex compact set that contains $K$ in its interior, such that we still have $K' \cap \Gamma = \varnothing$. Increasing $R$ if necessary, we may assume that $R \ge M$, where $M$ is determined by the $\C^*$-nice projection property for $\gamma_1,\dots,\gamma_m$. We may also assume that $K' \subset \C^* \times A_{R}$ and $\gamma_j(0) \in \C^* \times A_{R}$ for $j=1,\dots,m$. Choose some $R' > R$. Let $\tilde\Gamma = \Gamma \cap (\C^* \times \overline{A}_{R'}) = (\pi_2\rvert_\Gamma)^{-1}(\overline{A}_{R'})$. By the $\C^*$-nice projection property, $\tilde\Gamma$ is compact and consists of precisely $m$ components $\tilde\Gamma_1,\dots,\tilde\Gamma_m$, each $\tilde\Gamma_j = \Gamma_j \cap (\C^* \times\overline{A}_{R'})$ a smoothly embedded compact curve. We first show how to use Theorem~\ref{thm:andersenlempert} to construct a continuous family of volume preserving automorphisms $\alpha_t \in \Aut((\C^*)^2)$ starting at the identity and satisfying the following conditions.
\begin{enumerate}
 \item $\sup_{x \in K'}d(\alpha_t(x), x) < \epsilon/2$ for all $t \in I$.
 \item $\alpha_1(\tilde\Gamma) \subset (\C^*)^2 \setminus \overline{P}_R$.
\end{enumerate}

Let $U_0$ be a neighbourhood of $K'$ disjoint from $\tilde\Gamma$. For $j = 1,\dots,m$ choose a point $p_j \in \tilde\Gamma_j$ satisfying $p_j \notin \C^* \times\overline{A}_R$. Applying Lemma \ref{lem:shrinkcurve} to each compact embedded curve $\tilde\Gamma_j$ gives $\CC^1$-isotopies of injective volume preserving holomorphic maps $\chi_{j,t} : U_j \to (\C^*)^2$, where each $U_j$ is a small open neighbourhood of $\tilde\Gamma_j$, such that each $\chi_{j,0}$ is the inclusion and $\chi_{j,1}(\tilde\Gamma_j) \subset (\C^*)^2 \setminus (\C^* \times\overline{A}_R)$ for $j = 1,\dots,m$. By Lemma \ref{lem:shrinkcurve} we can also ensure that the sets $U_0, \chi_{1,t}(U_1), \dots, \chi_{m,t}(U_m)$ are pairwise disjoint for all $t \in I$. Noting that the disjoint union of $K'$ with finitely many embedded compact curves is $\OO((\C^*)^2)$-convex (see \cite[Lemma~3]{Ritter2013}), we now apply Theorem~\ref{thm:andersenlempert} to the $\CC^1$-isotopy of injective volume preserving holomorphic maps defined on $U_0 \cup U_1 \cup \dots \cup U_m$ that equals the inclusion on $U_0$ for all $t \in I$, and equals $\chi_{j,t}$ on $U_j$, $j = 1,\dots,m$, thereby obtaining the desired family $\alpha_t \in \Aut((\C^*)^2)$. In doing so, we may assume that $U_1,\dots,U_m$ are all contractible and, as explained in Remark \ref{rem:andersenlempert}, we do not require any cohomological assumptions on $U_0$, where we approximate the inclusion.

Now let $\Gamma' = \Gamma \cap (\C^* \times \overline{A}_{R})$.
Although the automorphism $\alpha_1$ moves all of $\tilde\Gamma$, and hence all of $\Gamma' \subset \tilde\Gamma$, outside of $\overline{P}_R$, it may move parts of $\Gamma\setminus\Gamma'$ into $\overline{P}_R$ that were not there before. By following the arguments in \cite{LarussonRitter2012,Ritter2013} (with the obvious modifications necessary due to working in $(\C^*)^2$ rather than $\C\times\C^*$) we construct a continuous family of automorphisms $\beta_t \in \Aut((\C^*)^2)$ starting at the identity, given by $\beta_t(z,w) = (ze^{tg(w)}, w)$, where $g \in \OO(\C^*)$ is obtained by a suitable application of Mergelyan's theorem. Here, we make use of special shear automorphisms of $(\C^*)^2$ to give an explicit formula for the family $\beta_t$, which has the following properties.
\begin{enumerate}
 \item $\sup\limits_{x \in K \cup \Gamma'} d(\beta_t(x), x) < \epsilon/2$ for all $t \in I$.
 \item $\beta_1(\Gamma) \cap \alpha_1^{-1}(\overline{P}_R) = \varnothing$.
\end{enumerate}
After possibly shrinking $\epsilon$ further, the approximation of the identity by $\beta_t$ is sufficiently good on $K \cup \Gamma'$ that the composition $\theta_t = \alpha_t \circ \beta_t$ satisfies all the required conditions, except for fixing the points $q_1,\dots,q_n$.

To ensure that each $\theta_t(q_k) = q_k$ for $k = 1, \dots, n$, begin by assuming the points $q_k$ are in generic position, in the sense that $\pi_i(q_k) \neq \pi_i(q_{k'})$ for $1\le k <  k' \le m$ and $i=1,2$. (If this is not the case, then we can make it so by conjugating by a composition of shear automorphisms of $(\C^*)^2$ that approximates the identity uniformly on a large compact set.) By shrinking $\epsilon$ if necessary, the above argument yields a family $\theta_t \in \Aut((\C^*)^2)$ sufficiently close to the identity on $K$ that the points $\theta_t(q_k)$, $k = 1,\dots,n$, remain in generic position for all $t \in I$ (recall that we assume $q_1, \dots, q_n \in K$). Following the argument at the end of the proof of Proposition~\ref{prop:niceprojectionproperty}, we then construct a continuous family of shear automorphisms $\mu_t \in \Aut((\C^*)^2)$, starting at the identity automorphism, that adjusts the second coordinate of each point $q_k$ (depending on the first coordinate only) so that $\pi_2(\mu_t \circ \theta_t (q_k)) = \pi_2(q_k)$ for all $t \in I$. Similarly, we construct $\nu_t \in \Aut((\C^*)^2)$ to adjust the first coordinates, giving $\pi_1(\nu_t \circ \mu_t \circ \theta_t(q_k)) = \pi_1(q_k)$ for all $t \in I$. Noting that $\nu_t$ is the identity in the second component, we have $\nu_t \circ \mu_t \circ \theta_t(q_k) = q_k$ for $k = 1, \dots, n$. By shrinking $\epsilon$ sufficiently small, we ensure that both $\mu_t$ and $\nu_t$ may approximate the identity sufficiently closely on a large compact set that the family $\nu_t \circ \mu_t \circ \theta_t \in \Aut((\C^*)^2)$ satisfies all conclusions of the theorem.
\end{proof}

At this point we wish to remark that, given the above lemma, the proof of the Wold embedding theorem given in \cite[Theorem~1]{Ritter2013} applies without change in the current situation, giving the corresponding Wold embedding theorem for $(\C^*)^2$. The same remark applies to \cite[Lemma~6]{Ritter2013} on the homotopy class of the resulting embedding. Combined with Propositions \ref{prop:exposeboundarypoints} and \ref{prop:niceprojectionproperty} above we obtain the following version of Forstneri\v c and Wold's result on properly embedding the interior of a compact bordered Riemann surface (see \cite[Corollary~1.2]{ForstnericWold2009} and \cite[Theorem~4]{Ritter2013}).

\begin{theorem}
\label{thm:embeddingcompactbordered}
Let $L$ be a compact bordered Riemann surface and $f : L \to (\C^*)^2$ be an embedding. Then $f$ can be approximated, uniformly on compact subsets of $\mathring{L}$, by proper embeddings $\mathring{L} \to (\C^*)^2$ that are homotopic to $f\rvert_{\mathring{L}}$.
\end{theorem}

\section{Enlarging the domain of an embedding with control over the image}
\label{sec:enlargingdomainofembedding}

\noindent
Let $X$ be an open Riemann surface containing a compact bordered Riemann surface $L \subset X$. Given a continuous map of $X$ into $(\C^*)^2$ that restricts to a holomorphic embedding in a neighbourhood of $L$ (see Definition \ref{def:embedding}), we now use the results from the previous section to deform the map to a holomorphic embedding of a slightly larger compact bordered Riemann surface $L'$ that contains $L$ in its interior, with certain control over the images of $bL'$ and $L' \setminus \mathring{L}$. This will in turn allow us to construct proper embeddings in the following section. This idea was introduced by Alarc\'on and L\'opez in their recent paper \cite{AlarconLopez2013}. We make the following definition.

\begin{definition}
\label{def:diffeotopy}
Let $X$ be a smooth manifold. A \emph{diffeotopy} of $X$ is a $\CC^\infty$-isotopy $\sigma_t:X \to X$ of smooth diffeomorphisms of $X$ such that $\sigma_0 = \id_X$.
\end{definition}

\begin{lemma}
\label{lem:enlargingdomainofembedding}
Let $X$ be an open Riemann surface and $L \subset X$ be a compact bordered Riemann surface. Let $f : X \to (\C^*)^2$ be a continuous map that restricts to a holomorphic embedding of an open neighbourhood of $L$ and satisfies $f(bL) \subset (\C^*)^2 \setminus \overline{P}_r$ for some $r > 0$. Given an open neighbourhood $U$ of $bL$, a discrete set $S \subset X$ satisfying $S \cap bL = \varnothing$, $\epsilon > 0$, and $R > r$, there exists a compact bordered Riemann surface $L' \subset X$ containing $L$ in its interior, a homotopy of continuous maps $f_t : X \to (\C^*)^2$, and a diffeotopy $\sigma_t$ of $X$ together satisfying the following properties.
\begin{enumerate}
 \item $f_0 = f$.
 \item There exists an open neighbourhood of $L$ on which $f_t : X \to (\C^*)^2$ restricts to a holomorphic embedding for all $t \in I$.
 \item $\sup_{x \in L}d(f_t(x),f(x)) < \epsilon$ for all $t \in I$.
 \item $f_t\rvert_S = f\rvert_S$ for all $t \in I$.
 \item $f_1 : X \to (\C^*)^2$ restricts to a holomorphic embedding of an open neighbourhood of $L'$.
 \item $f_1(bL') \subset (\C^*)^2 \setminus \overline{P}_R$.
 \item $f_1(L'\setminus \mathring{L}) \subset (\C^*)^2 \setminus \overline{P}_r$.
 \item $\sigma_1(L') = L$.
 \item \label{property:restrictidentity} $\sigma_t$ restricts to the identity on $X \setminus U$, for all $t \in I$.
\end{enumerate}
\end{lemma}

\begin{proof}
Let $r' > r$ be chosen slightly larger so that we still have $f(bL) \subset (\C^*)^2 \setminus \overline{P}_{r'}$. Since $S$ is discrete and disjoint from $bL$, we may shrink $U$ about $bL$ so that $S \cap \overline{U} = \varnothing$. Note that by property (\ref{property:restrictidentity}) this will ultimately ensure that $\sigma_t$ is the identity on a fixed neighbourhood of $S$. We may also assume that $U$ is sufficiently small so that $f$ restricts to a holomorphic embedding on $L \cup U$ and that $f(U) \subset (\C^*)^2 \setminus \overline{P}_{r'}$. By further shrinking $U$ we can assume that $L$ has a smooth defining function $\tau : U \to \R$ without critical points such that $L \cap U = \{\tau \le 0\}$. (To see this, note that since $X$ is 1-dimensional, $L$ is trivially strictly pseudoconvex, and therefore has a global defining function for the entire boundary.) Scaling $\tau$ by a positive constant we obtain a family $L_t = L \cup \{\tau \le t\}$, $t \in I$, of compact bordered Riemann surfaces with boundaries in $U$ such that $L = L_0$ and $L_t \subset \mathring{L}_{t'}$ for $0\le t < t' \le 1$, and $L_t$ is $\OO(\mathring{L}_{t'})$-convex. Furthermore, using the negative gradient flow of $\tau$ multiplied by a suitable cutoff function with support in $U$, we obtain a diffeotopy $\sigma_t$ of $X$ that equals the identity outside of $U$ and satisfies $\sigma_t(L_t) = L_0$ for all $t \in I$ (this is a standard argument in Morse theory, see for example \cite[Theorem~3.1]{Milnor1963}).

Consider the compact bordered Riemann surface $L_1$. Recall that $f$ restricts to an embedding on $L \cup U$, and is therefore an embedding of an open neighbourhood of $L_1$. For each boundary component $\gamma_j$ of $L_1$, $j = 1,\dots,m$, choose a point $a_j \in \gamma_j$. For each $j = 1,\dots,m$, let $U_j$ be a neighbourhood of $a_j$ in $X$ that does not meet $L$ and such that $f(\overline{U}_j) \subset (\C^*)^2\setminus\overline{P}_{r'}$. Since $S$ is discrete it meets $L_1$ in finitely many points, $S\cap L_1 = S \cap L = \{b_1,\dots,b_n\}$ (recall that $S$ does not meet $U$). We now apply Proposition~\ref{prop:exposeboundarypoints} to obtain a homotopy $f_t : L_1 \to (\C^*)^2$ satisfying the following conditions.
\begin{itemize}
 \item $f_0 = f\rvert_{L_1}$.
 \item $f_t(b_k) = f(b_k)$ for all $t \in I$ and all $k = 1,\dots,n$.
 \item $\sup_{x\in L}d(f_t(x),f(x)) < \epsilon/3$ for all $t \in I$.
 \item $f_t(L_1 \setminus \mathring{L}) \subset (\C^*)^2\setminus\overline{P}_{r'}$ for all $t \in I$ (see Remark \ref{rem:exposeboundarypoints}).
 \item For all $t \in I$, $f_t$ restricts to a holomorphic embedding of some fixed open neighbourhood of $L$.
 \item $f_1 : L_1 \to (\C^*)^2$ is a holomorphic embedding such that $a_1,\dots,a_m$ are $f_1$-exposed.
 \item $\pi_1(f_1(a_j)) \notin \overline{A}_{r'}$ for $j = 1,\dots,m$.
\end{itemize}

If necessary, shrink the sets $U_j$ so that $\pi_1(f_1(\overline{U}_j)) \subset \C^* \setminus \overline{A}_{r'}$ for $j = 1,\dots,m$.

We next apply Proposition \ref{prop:niceprojectionproperty} to the embedding $f_1 : L_1 \to (\C^*)^2$, giving a homotopy $g_t : L_1 \setminus \{a_1,\dots,a_m\} \to (\C^*)^2$ satisfying the following conditions.
\begin{itemize}
 \item $g_0 = f_1\rvert_{L_1 \setminus \{a_1,\dots,a_m\}}$.
 \item $g_t(b_k) = f_1(b_k) = f(b_k)$ for all $t \in I$ and all $k = 1,\dots,n$.
 \item $\sup_{x \in L} d(g_t(x),f_1(x)) < \epsilon/3$ for all $t \in I$.
 \item $g_t(L_1 \setminus (\{a_1,\dots,a_m\} \cup \mathring{L})) \subset (\C^*)^2\setminus\overline{P}_{r'}$ for all $t \in I$ (see Remark \ref{rem:niceprojectionproperty}).
 \item For all $t\in I$, $g_t$ restricts to a holomorphic embedding of some fixed open neighbourhood of $L$.
 \item $g_1 : L_1 \setminus \{a_1,\dots,a_m\} \to (\C^*)^2$ is a holomorphic embedding such that $g_1(bL_1\setminus\{a_1,\dots,a_m\})$ satisfies the $\C^*$-nice projection property.
\end{itemize}


Now note that $g_1(\mathring{L}_1) \subset (\C^*)^2$ is a (non-properly) embedded open Riemann surface that is the interior of the embedded bordered Riemann surface $g_1(L_1 \setminus \{a_1,\dots,a_m\})$, all of whose boundary components are unbounded curves. Since $L \subset X$ is $\OO(\mathring{L}_1)$-convex and $g_1(bL) \cap \overline{P}_{r'} = \varnothing$, it is well known that $g_1(L) \cup \overline{P}_{r'} \subset (\C^*)^2$ is $\OO((\C^*)^2)$-convex (see for example the proof of \cite[Proposition~3.1]{Wold2006a} or \cite[Lemma~5]{Ritter2013}). We may therefore apply Lemma~\ref{lem:woldlemma} to the set $g_1(L) \cup \overline{P}_{r'}$, the curves $g_1(bL_1\setminus\{a_1,\dots,a_m\})$, and the points $g_1(b_k) = f(b_k)$, $k=1,\dots,n$, to obtain a continuous family $\theta_t \in \Aut((\C^*)^2)$ satisfying the conditions as stated in the lemma. In particular, note that if the approximation of the identity by $\theta_t$ is sufficiently close on $\overline{P}_{r'}$ then no points from outside $\overline{P}_{r'}$ are moved into $\overline{P}_r$ by $\theta_t$, that is, we have $\theta_t((\C^*)^2 \setminus \overline{P}_{r'})\subset (\C^*)^2 \setminus \overline{P}_{r}$ for all $t \in I$. Thus $\theta_t(g_1(L_1 \setminus (\{a_1,\dots,a_m\} \cup \mathring{L}))) \subset (\C^*)^2 \setminus \overline{P}_{r}$ for all $t \in I$. Letting $h_t = \theta_t \circ g_1 : L_1 \setminus \{a_1,\dots,a_m\} \to (\C^*)^2$, we see that the following properties hold for $h_t$.
\begin{itemize}
 \item $h_0 = g_1$.
 \item $h_t(b_k) = g_1(b_k) = f(b_k)$ for all $t \in I$ and all $k = 1,\dots,n$.
 \item $\sup_{x \in L}d(h_t(x),g_1(x)) < \epsilon/3$.
 \item For all $t \in I$, $h_t$ restricts to a holomorphic embedding of some fixed open neighbourhood of $L$.
 \item $h_1 : L_1 \setminus \{a_1,\dots,a_m\} \to (\C^*)^2$ is a holomorphic embedding.
 \item $h_1(L_1 \setminus (\{a_1,\dots,a_m\} \cup \mathring{L})) \subset (\C^*)^2 \setminus \overline{P}_{r}$.
 \item $h_1(bL_1\setminus\{a_1,\dots,a_m\}) \subset (\C^*)^2 \setminus \overline{P}_R$.
\end{itemize}

We now combine the three homotopies $f_t\rvert_{L_1 \setminus \{a_1,\dots,a_m\}}, g_t$, and $h_t$ into a single homotopy by matching at corresponding endpoints, and then reparameterise to obtain a single homotopy that we denote by $\tilde{f}_t : L_1 \setminus \{a_1,\dots,a_m\} \to (\C^*)^2$, $t \in I$.

Consider $L_{t'} \subset \mathring{L}_1$ for some suitably chosen $t' < 1$ very close to $1$ and let $L' = L_{t'}$. By taking $t'$ sufficiently close to 1, we can ensure that $\tilde{f}_1(bL') \subset (\C^*)^2 \setminus \overline{P}_R$. Now let $\chi : X \to [0,1]$ be a smooth cutoff function with support in $\mathring{L}_1$ that equals $1$ in a neighbourhood of $L'$. Let $f_t(x) = \tilde{f}_{\chi(x)\cdot t}(x)$. Since $\tilde{f}_0 = f\rvert_{L_1 \setminus \{a_1,\dots,a_m\}}$, we may extend $f_t$ outside of the support of $\chi$ by the initial function $f : X \to (\C^*)^2$ to obtain a homotopy of continuous maps that we still denote by $f_t : X \to (\C^*)^2$. Reparameterising $\sigma_t$ so that $\sigma_1(L') = L_0 = L$, we see that $L'$, $f_t$, and $\sigma_t$ satisfy all conclusions of the lemma.
\end{proof}

\section{Main theorem}
\label{sec:maintheorem}

\noindent
Before stating our main result we give the following lemma, which illustrates how we will construct a $\CC^\infty$-isotopy of smooth embeddings $\Phi_t : X \to X$ of the open Riemann surface $X$ into itself, as required in Theorems~\ref{thm:simplifiedmaintheorem} and \ref{thm:mainresult}.

\begin{lemma}
\label{lem:deformationunion}
Let $X$ be a smooth manifold. Let $K_0 \subset K_1 \subset \dots$ and $L_0 \subset L_1 \subset \dots$ be increasing sequences of compact sets in $X$ such that $K_{j} \subset K_{j+1}^\circ$ and $L_{j} \subset L_{j+1}^\circ$ for all $j = 0, 1, \dots$. Let $M = \bigcup\limits_{j=0}^\infty K_j$ and $N = \bigcup\limits_{j=0}^{\infty}L_j$. Suppose that there exists a sequence of diffeotopies $\phi_{j,t} : X \to X$, $t \in I$, $j = 0,1,\dots$ (see Definition \ref{def:diffeotopy}) such that the following conditions hold.
\begin{enumerate}
 \item $\phi_{j,1}(K_j) = L_j$ for $j = 0,1,\dots$.
 \item $L_{j-1} \subset \phi_{j,1}(K_{j-1})$ for $j = 1, 2, \dots$.
 \item $\phi_{j+1,t}(x) = \phi_{j,t}(x)$ for all $x \in K_{j-1}$ and all $t \in I$, for $j = 1,2,\dots$.
\end{enumerate}
Then $\Phi_t: M \to X$ defined by
\[
  \Phi_t(x) = \lim_{j\to\infty} \phi_{j,t}(x), \quad x\in M, t\in I,
\]
gives a $\CC^\infty$-isotopy of smooth embeddings of $M$ into $X$ such that $\Phi_0$ is the inclusion $M \hookrightarrow X$ and $\Phi_1(M) = N$.
\end{lemma}

\begin{proof}
First note that both $M$ and $N$ are open submanifolds of $X$. If $M$ is empty the lemma is vacuous, so assume $M \neq \varnothing$. It is clear that $\Phi_t$ is a well-defined smooth isotopy of injective immersions of $M$, which are then necessarily embeddings since the dimensions of the source and target agree. Let $x \in M$, so that $x \in K_j$ for some $j$. Then 
\[\Phi_1(x) = \phi_{j+1,1}(x) \in \phi_{j+1,1}(K_{j+1}) = L_{j+1} \subset N\,.\]
Conversely, let $y \in N$. For some $j$ we then have, applying condition (2),
\[y \in L_j \subset \phi_{j+1,1}(K_j) = \Phi_1(K_j) \subset \Phi_1(M)\,,\]
thereby completing the proof.
\end{proof}
 
We now state and prove our main theorem. We apply the preceding lemma in the case that $X$ is an open Riemann surface and $\bigcup\limits_{j=0}^{\infty}K_j = X$ to obtain a $\CC^\infty$-isotopy of smooth embeddings $\Phi_t : X \to X$ such that $\Phi_1(X)$ properly holomorphically embeds into $(\C^*)^2$.

\begin{theorem}
\label{thm:mainresult}
Let $X$ be an open Riemann surface, $K \subset X$ be an $\OO(X)$-convex compact set, and $S \subset X$ be a discrete set. Let $f : X \to (\C^*)^2$ be a continuous map that restricts to a holomorphic embedding of some open neighbourhood of $K$, and such that $f\rvert_S : S \to (\C^*)^2$ is a proper injection. Given $\epsilon > 0$, there exists a $\CC^\infty$-isotopy of smooth embeddings $\Phi_t : X \to X$, $t \in I$, and a continuous family of continuous maps $f_t : \Phi_t(X) \to (\C^*)^2$, $t \in I$, satisfying the following conditions.
\begin{enumerate}
 \item $\Phi_0 = \id_X$.
 \item \label{thm:embedding_approx:inclusion} $\Phi_t\rvert_V = \id_V$ for all $t \in I$, for some fixed open neighbourhood $V\subset X$ of $K \cup S$.
 \item $f_0 = f : X \to (\C^*)^2$.
 \item \label{thm:embedding_approx:embedding_nbhd} $f_t\rvert_U : U \to (\C^*)^2$ is a holomorphic embedding for all $t \in I$, for some fixed open neighbourhood $U \subset V$ of $K$.
 \item \label{thm:embedding_approx:approximation} $\sup\limits_{x \in K} d(f_t(x),f(x)) < \epsilon$ for all $t \in I$.
 \item \label{thm:embedding_approx:interpolation} $f_t\rvert_S = f\rvert_S$ for all $t \in I$.
 \item $f_1 : \Phi_1(X) \to (\C^*)^2$ is a proper holomorphic embedding.
\end{enumerate}
\end{theorem}

\begin{proof}
First note that if $K = \varnothing$ we may deform $f$ on a neighbourhood of a small compact disc $D \subset X \setminus S$ to a map that restricts to a holomorphic embedding on a neighbourhood of $D$, without changing the values of $f$ on $S$. Thus we may assume that $K \neq \varnothing$.

We now proceed to first prove the result in the case that $S = \varnothing$. Following this, we explain the necessary additions and modifications required when $S \neq \varnothing$.

Suppose $S = \varnothing$. Let $W$ be an open neighbourhood of $K$ on which $f$ restricts to a holomorphic embedding $f\rvert_W : W \to (\C^*)^2$. Since $X$ is a 1-dimensional Stein manifold, there exists a smooth strictly subharmonic exhaustion function $\rho : X \to \R$ such that $K \subset \{\rho < 0\}$ and $X \setminus W \subset \{\rho > 0\}$. Without loss of generality we may assume that $\rho$ is a Morse function with $0$ as a regular value, and that the preimage of every critical value of $\rho$ contains a single critical point. Letting $K_0 = \{\rho \le 0\}$, we have $K \subset \mathring{K}_0 \subset K_0 \subset W$, so that $f$ still restricts to an embedding of an open neighbourhood of $K_0$.

Let $p_1, p_2, \dots$ be the critical points of $\rho$ outside $K_0$, ordered so that $\rho(p_1) < \rho(p_2) < \dots$. Let $0 = c_0 < c_1 < c_2 < \dots$ be a sequence of regular values of $\rho$ with $\lim\limits_{j \to \infty} c_j = \infty$, chosen so that $c_{2j-1} < \rho(p_j) < c_{2j}$ for $j = 1, 2, \dots$.
(If $\rho$ has only finitely many critical points $p_1, \dots, p_n$, choose the tail $c_{2n+1}, c_{2n+2}, \dots$ arbitrarily so that $\lim\limits_{j \to \infty} c_j = \infty$.) Letting $K_j = \{\rho \le c_j\}$, each $K_j$ is then an $\OO(X)$-convex compact bordered Riemann surface satisfying $K_j \subset \mathring{K}_{j+1}$, and we have $\bigcup\limits_{j=0}^\infty K_j = X$. 

Choose a regular value $c_{-1} < 0$ of $\rho$ and set $K_{-1} = \{\rho \le c_{-1}\}$. By taking $c_{-1}$ sufficiently close to $0$, we may assume $K \subset \mathring{K}_{-1} \subset K_{-1} \subset \mathring{K}_0$. We now construct an increasing sequence of compact bordered Riemann surfaces $L_{-1} \subset L_0 \subset L_1 \subset \dots$ with $L_{j-1} \subset L_{j}^\circ$ for all $j = 0, 1, \dots$, a sequence of smooth diffeotopies $\phi_{j,t} : X \to X$, $j = 0, 1, \dots$, and a sequence of homotopies of continuous maps $f_{j,t} : X \to (\C^*)^2$, $t \in I$, $j = 0, 1, \dots$, such that the following conditions hold for all $j = 0, 1, \dots$ (provided that they make sense).

\begin{enumerate}[(i)]
 \item \label{proof:embedding_approx:deform_compact}$\phi_{j,1}(K_j) = L_j$.
 \item \label{proof:embedding_approx:keep_controlled}$L_{j-1} \subset \phi_{j,1}(K_{j-1})$.
 \item \label{proof:embedding_approx:matching} $\phi_{j,t}(x) = \phi_{j-1,t}(x)$ for all $x \in K_{j-2}$ and all $t \in I$.
 \item \label{proof:embedding_approx:matching_endpoints} $f_{j,0} = f_{j-1,1}$.
 \item \label{proof:embedding_approx:approximation} $\sup\limits_{x \in L_{j-1}}d(f_{j,t}(x),f_{j,0}(x)) < \epsilon_j$ for some small $\epsilon_j > 0$ to be specified, for all $t \in I$.
 \item \label{proof:embedding_approx:hol_embedding} There exists a neighbourhood of $L_{j-1}$ on which $f_{j,t} : X \to (\C^*)^2$ restricts to a holomorphic embedding for each $t \in I$.
 \item \label{proof:embedding_approx:hol_embedding_final} There exists a neighbourhood of $L_j$ on which $f_{j,1} : X \to (\C^*)^2$ restricts to a holomorphic embedding.
 \item \label{proof:embedding_approx:boundary_far_away} $f_{j,1}(bL_j) \subset (\C^*)^2 \setminus \overline{P}_j$.
 \item \label{proof:embedding_approx:collar_far_away} $f_{j,1}(L_j \setminus \mathring{L}_{j-1}) \subset (\C^*)^2 \setminus \overline{P}_{j-1}$.
\end{enumerate}

Assuming this construction complete, we now show that the theorem holds. Let $\Phi_t = \lim\limits_{j\to\infty}\phi_{j,t}$. By Lemma \ref{lem:deformationunion} and conditions (\ref{proof:embedding_approx:deform_compact})--(\ref{proof:embedding_approx:matching}), $\Phi_t : X \to X$ is a $\CC^\infty$-isotopy of smooth embeddings satisfying $\Phi_0 = \id_X$ and $\Phi_1(X) = \bigcup\limits_{j=0}^\infty L_j$. We will shortly see that for $j=0$ we may choose $\phi_{0,t} = \id_X$ for all $t \in I$. Then the fact that $K \subset \mathring{K}_{-1}$, together with property (\ref{proof:embedding_approx:matching}), gives condition (\ref{thm:embedding_approx:inclusion}) of the theorem with $V = \mathring{K}_{-1}$.
 
We will also see that we have $f_{0,0} = f$. At each step $j = 0, 1, \dots$ we reparametrise the family $f_{j,t}$ by scaling $t \in I$ so that we instead have $t \in 2^{-(j+1)}I$. Glue together each pair of homotopies $f_{j-1,t}$ and $f_{j,t}$ at the matching endpoints $f_{j-1,2^{-j}} = f_{j,0}$ so as to obtain a continuous family of continuous maps $f_t : X \to (\C^*)^2$, $t \in [0,1)$. Assuming that the $\epsilon_j$ are sufficiently small, properties (\ref{proof:embedding_approx:approximation}) and (\ref{proof:embedding_approx:hol_embedding_final}) show that $f_t$ converges as $t \to 1$ locally uniformly on $\bigcup\limits_{j=0}^\infty L_j = \Phi_1(X)$ to a holomorphic map $F : \Phi_1(X) \to (\C^*)^2$. For each $t \in [0,1)$, restrict $f_t$ to $\Phi_t(X) \subset X$. Setting $f_1 = F$ we obtain a continuous family of continuous maps $f_t : \Phi_t(X) \to (\C^*)^2$ such that $f_1$ is holomorphic.

Taking $L_{-1} = K_{-1}$, as we shall do, and recalling that $K \subset \mathring{K}_{-1}$, property (\ref{proof:embedding_approx:hol_embedding}) ensures that condition (\ref{thm:embedding_approx:embedding_nbhd}) in the theorem holds for $t \in [0,1)$. If the $\epsilon_j > 0$ are chosen sufficiently small, then property (\ref{proof:embedding_approx:approximation}) ensures that condition (\ref{thm:embedding_approx:approximation}) holds for all $t \in I$.

For each $j$, $f_{j,2^{-(j+1)}} = f_{j+1,0}$ is a holomorphic embedding of a neighbourhood of $L_j$ by property (\ref{proof:embedding_approx:hol_embedding_final}). Choosing the $\epsilon_j$ sufficiently small ensures that for each $j$, the uniform limit $f_1$ is sufficiently close to $f_{j+1,0}$ on $L_j$ to ensure that $f_1$ restricts to a holomorphic embedding of a neighbourhood of the smaller compact set $L_{j-1}$. Thus $f_1$ is a holomorphic embedding of $\bigcup\limits_{j=0}^\infty L_j = \Phi_1(X)$. Again, by taking the $\epsilon_j$ sufficiently small, condition (\ref{proof:embedding_approx:collar_far_away}) ensures that $f_1(L_j \setminus \mathring{L}_{j-1}) \subset (\C^*)^2\setminus\overline{P}_{j-1}$ holds for all $j = 0, 1, \dots$, implying that $f_1$ is a proper map. This completes the proof when $S = \varnothing$, assuming conditions (\ref{proof:embedding_approx:deform_compact})--(\ref{proof:embedding_approx:collar_far_away}) hold.

Let us now show that the conditions above can be ensured for $j = 0$. We let $\phi_{0,t} = \id_X$ for all $t \in I$, and set $L_{-1} = K_{-1}$, $L_0 = K_0$. Conditions (\ref{proof:embedding_approx:deform_compact}) and (\ref{proof:embedding_approx:keep_controlled}) hold, while (\ref{proof:embedding_approx:matching}) is vacuous. Letting $f_{0,t} = f$ for $t \in I$ and recalling that $f$ is an embedding in a neighbourhood of $K_0$ gives conditions (\ref{proof:embedding_approx:approximation})--(\ref{proof:embedding_approx:hol_embedding_final}), while conditions (\ref{proof:embedding_approx:matching_endpoints}), (\ref{proof:embedding_approx:boundary_far_away}), and (\ref{proof:embedding_approx:collar_far_away}) are vacuous.

Now fix some $j > 0$ and suppose the conditions are satisfied for $j-1$. Further suppose for the moment that there are no critical points of $\rho$ in the set $K_{j} \setminus \mathring{K}_{j-1}$. This is the \emph{non-critical case}, in which the topology of the sublevel sets $K_{j-1}$ and $K_j$ is the same. Note that all precisely all odd values of $j$ give inductive steps of non-critical type, unless $\rho$ has only finitely many critical points, in which case eventually every step is non-critical.

Note that $\phi_{j-1,1}(K_{j-2})$ is a compact subset of $\phi_{j-1,1}(\mathring{K}_{j-1}) = \mathring{L}_{j-1}$. Thus there exists $t_0 < 1$ close to $1$ such that $\cup_{t_0 \le t \le 1}\phi_{j-1,t}(K_{j-2})$ is also a compact subset of $\mathring{L}_{j-1}$. We now apply Lemma \ref{lem:enlargingdomainofembedding} to obtain a compact bordered Riemann surface with smooth boundary $L_j$ whose interior contains $L_{j-1}$, a diffeotopy $\sigma_{j,t} : X \to X$ satisfying $\sigma_{j,1}(L_{j}) = L_{j-1}$ with $\sigma_{j,t}$ equal to the identity on $\cup_{t_0 \le t \le 1}\phi_{j-1,t}(K_{j-2})$, and a homotopy $f_{j,t} : X \to X$ linking $f_{j-1,1} = f_{j,0}$ to some continuous map $f_{j,1} : X \to (\C^*)^2$. The homotopy $f_{j,t}$ and its final map $f_{j,1}$ satisfy conditions (\ref{proof:embedding_approx:approximation})--(\ref{proof:embedding_approx:collar_far_away}) above.

Note that $\sigma_{j,1}(L_{j-1})$ is a compact subset of $\sigma_{j,1}(\mathring{L}_j) = \mathring{L}_{j-1}$, so that $\phi_{j-1,1}^{-1}(\sigma_{j,1}(L_{j-1}))$ is a compact subset of $\phi_{j-1,1}^{-1}(\mathring{L}_{j-1}) = \mathring{K}_{j-1}$. Because $\rho$ is a Morse function and there are no critical points in the compact set $K_j \setminus \mathring{K}_{j-1}$, there exists a diffeotopy $\psi_{j,t} : X \to X$ such that $\psi_{j,1}(K_j) = K_{j-1}$, where we may choose $\psi_{j,t}$ to equal the identity on the compact set $\phi_{j-1,1}^{-1}(\sigma_{j,1}(L_{j-1})) \cup K_{j-2} \subset \mathring{K}_{j-1}$.

We now define $\phi_{j,t} : X \to X$ as follows.
\begin{equation*}\label{formula:diffeotopy}
 \phi_{j,t} = \left\{
	\begin{array}{ll}
		\phi_{j-1,t} \circ \psi_{j,t} & \text{for } 0 \le t \le t_0,\\
		\sigma_{j,t}^{-1} \circ \phi_{j-1,t} \circ \psi_{j,t} & \text{for } t_0 \le t \le 1,
	\end{array} \right.\tag{\dag}
\end{equation*}
where we have reparameterised the diffeotopy $\sigma_{j,t}$ in the $t$-variable so that $t \in [t_0,1]$, with $\sigma_{j,t_0} = \id_X$, $\sigma_{j,1}(L_j) = L_{j-1}$, and so that $\phi_{j,t}$ as defined above is smooth at $t_0$.

We now show that conditions (\ref{proof:embedding_approx:deform_compact})--(\ref{proof:embedding_approx:matching}) hold. We have
\[
 \phi_{j,1}(K_j) = \sigma_{j,1}^{-1} \circ \phi_{j-1,1} \circ \psi_{j,1}(K_j) = \sigma_{j,1}^{-1} \circ \phi_{j-1,1}(K_{j-1}) = \sigma_{j,1}^{-1}(L_{j-1}) = L_j\,,
\]
so condition (\ref{proof:embedding_approx:deform_compact}) holds.

Since $\psi_{j,t}$ is the identity on $\phi_{j-1,1}^{-1}(\sigma_{j,1}(L_{j-1})) \subset K_{j-1}$, applying $\psi_{j,1}$ to both sides gives $\phi_{j-1,1}^{-1}(\sigma_{j,1}(L_{j-1})) \subset \psi_{j,1}(K_{j-1})$, so that $L_{j-1} \subset \sigma_{j,1}^{-1}\circ\phi_{j-1,1}\circ\psi_{j,1}(K_{j-1}) = \phi_{j,1}(K_{j-1})$, showing that condition (\ref{proof:embedding_approx:keep_controlled}) holds.

Let $x \in K_{j-2}$. Suppose that $0 \le t \le t_0$. Then $\phi_{j,t}(x) = \phi_{j-1,t} \circ \psi_{j,t}(x) = \phi_{j-1,t}(x)$ as required, since $\psi_{j,t}$ equals the identity on $K_{j-2}$. On the other hand, if $t_0 \le t \le 1$, we still have $\phi_{j,t}(x) = \sigma_{j,t}^{-1} \circ \phi_{j-1,t} \circ \psi_{j,t}(x) = \sigma_{j,t}^{-1} \circ \phi_{j-1,t}(x) = \phi_{j-1,t}(x)$, since $\phi_{j-1,t}(x) \in \cup_{t_0 \le t \le 1}\phi_{j-1,t}(K_{j-2})$, a set on which $\sigma_{j,t}$ was chosen to be the identity. Condition (\ref{proof:embedding_approx:matching}) holds. This completes the proof for every non-critical step.

Now suppose that the conditions are satisfied for some $j-1$, but that there exists a single critical point $q$ of $\rho$ in the set $K_j \setminus \mathring{K}_{j-1}$. This is called the \emph{critical case}, and if $\rho$ has infinitely many critical points then for every even value of $j$ the inductive step will be critical. (If $\rho$ has $n < \infty$ critical points, then only for $j = 2, 4, \dots, 2n$ will the step be of critical type.) Since $\rho$ is subharmonic, the Morse index of $\rho$ at $q$ can only be either $0$ or $1$. We deal with these two cases in turn.

Suppose that the index of $\rho$ at $q$ is $0$. Then $K_j$ is diffeomorphic to the disjoint union of $K_{j-1}$ and a small closed disc $D \subset X$ that contains $q$ in its interior. More precisely, there exists a compact bordered Riemann surface $K'_j \subset X$ containing $K_{j-1}$ in its interior, a compact disc $D \subset X$ disjoint from $K'_j$, and a diffeotopy $\psi_{j,t}$ of $X$, such that $K_j = K'_j \cup D$ and $\psi_{j,1}(K'_j) = K_{j-1}$. Ignoring $D$ for the moment, we use the exact same procedure as for the non-critical case to find a compact bordered Riemann surface with smooth boundary $L'_j$ containing $L_j$ in its interior, a diffeotopy $\phi_{j,t}$ of $X$, and a homotopy $f'_{j,t} : X \to (\C^*)^2$, together satisfying the required conditions with respect to $K'_j$ and $L'_j$. Now let $\tilde{D} = \phi_{j,1}(D)$. Set $L_j = L'_j \cup \tilde{D}$, and extend the homotopy $f'_{j,t}$ for $t \in [1,2]$ by deforming $f'_{j,1}$ in a small neighbourhood of $\tilde{D}$ disjoint from $L'_j$ so that the final map $g_j$ is an embedding in a neighbourhood of $\tilde{D}$ satisfying $g_j(\tilde{D}) \subset (\C^*)^2 \setminus \overline{P}_j$. Call this extended homotopy $f_{j,t}$, and reparameterise so that $t \in [0,1]$. Then the conditions all hold for $K_j$, $L_j$, $\phi_{j,t}$ and $f_{j,t}$.

Now consider the case when the index of $\rho$ at $q$ is $1$. In this case, $K_j$ is obtained by adding a 1-handle to $K_{j-1}$. That is, there is a smoothly embedded 1-cell $E \subset \mathring{K}_j$ passing through $q$, with $E \cap K_{j-1} = bE$, such that $K_j$ deformation retracts to $K_{j-1} \cup E$. Furthermore, given any neighbourhood $W$ of $E$ we may find a 1-handle $H \subset K_j$ attached to $bK_{j-1}$, $E \subset H \subset W$, such that $K'_j := K_{j-1} \cup H$ is a compact bordered Riemann surface and there exists a diffeotopy $\psi_{j,t}$ of $X$ satisfying $\psi_{j,1}(K_j) = K'_j$. Here, $\psi_{j,t}$ can be chosen to be the identity outside of a neighbourhood of $K_j \setminus \mathring{K'_j}$.

Recall that the diffeotopy $\phi_{j-1,t}$ of $X$ satisfies $\phi_{j-1,1}(K_{j-1}) = L_{j-1}$. Let $\tilde{E} = \phi_{j-1,1}(E)$, a 1-cell satisfying $\tilde{E} \cap L_{j-1} = b\tilde{E}$. Recalling that $f_{j-1,1}$ is an embedding of a neighbourhood of $L_{j-1}$ satisfying $f_{j-1,1}(bL_{j-1}) \subset (\C^*)^2 \setminus \overline{P}_{j-1}$, we see that the endpoints of $\tilde{E}$ satisfy $f_{j-1,1}(b\tilde{E}) \subset (\C^*)^2 \setminus \overline{P}_{j-1}$. There is no obstruction to continuously deforming $f_{j-1,1}$ in a neighbourhood of $\tilde{E}$, while keeping it fixed on a neighbourhood of $L_{j-1}$ (and thus also fixed near $b\tilde{E}$), to obtain a map $g_{j-1} : X \to (\C^*)^2$ that remains a holomorphic embedding of a neighbourhood of $L_{j-1}$, but is now a smooth embedding on $\tilde{E}$ satisfying $g_{j-1}(\tilde{E}) \subset (\C^*)^2 \setminus \overline{P}_{j-1}$.
We now use Mergelyan's theorem to approximate $g_{j-1}$, uniformly on a neighbourhood of $L_{j-1}$ and in the $\CC^1$-topology on $\tilde{E}$, by a function $\tilde{g}_{j-1}$ that is holomorphic in a neighbourhood of $L_{j-1} \cup \tilde{E}$. If the approximation is sufficiently good then $\tilde{g}_{j-1}$ restricts to a holomorphic embedding of some smaller neighbourhood of $L_{j-1} \cup \tilde{E}$, on which it will be homotopic to $g_{j-1}$ through holomorphic embeddings (via a convex linear combination). Using a suitable cutoff function in the homotopy parameter that equals $1$ in an even smaller neighbourhood of $L_{j-1} \cup \tilde{E}$, we obtain a homotopy from $g_{j-1}$ to some map $\tilde{g}'_{j-1} : X \to (\C^*)^2$ which equals $\tilde{g}_{j-1}$ in a neighbourhood of $L_{j-1} \cup \tilde{E}$.

For a sufficiently small handle $H$, attached to $K_{j-1}$ and containing $E$, the set $\tilde{H} = \phi_{j-1,1}(H)$ will be a handle attached to $L_{j-1}$, containing $\tilde{E}$, and contained in the set on which $\tilde{g'}_{j-1}$ is a holomorphic embedding. Furthermore, we may assume that $\tilde{g'}_{j-1}(\tilde{H}) \subset (\C^*)^2 \setminus \overline{P}_{j-1}$. Setting $L'_j = \phi_{j-1,1}(K'_j) = L_{j-1} \cup \tilde{H}$ we also see that $\tilde{g'}_{j-1}(bL'_j) \subset (\C^*)^2 \setminus \overline{P}_{j-1}$.
Applying Lemma \ref{lem:enlargingdomainofembedding} to $L'_j$ and the map $\tilde{g}'_{j-1} : X \to (\C^*)^2$ gives a compact smoothly bordered Riemann surface $L_j$ containing $L'_j$ in its interior, a diffeotopy $\sigma_{j,t}$ of $X$ satisfying $\sigma_{j,1}(L_j) = L'_j$, and a homotopy $f'_{j,t}$ from $\tilde{g}'_{j-1}$ to a map $g_j : X \to (\C^*)^2$ satisfying all required conditions. Combining the homotopies above with $f'_{j,t}$ and reparameterising gives a homotopy $f_{j,t}$ from $f_{j-1,1}$ to $g_j$ satisfying all required conditions. We may now repeat the same construction as in the non-critical case to give a diffeotopy $\phi_{j,t}$ of $X$ so that all conditions hold for $K_j$, $L_j$, $\phi_{j,t}$ and $f_{j,t}$.

Now suppose that $S \neq \varnothing$ is a discrete set and that $f\rvert_S : S \to (\C^*)^2$ is a proper injection. We assume that $S \subset X$ is an infinite set; the case when $S$ is finite also follows from the following argument. As in the case when $S = \varnothing$, we obtain a smooth strictly subharmonic Morse exhaustion function $\rho : X \to \R$ with $0$ as a regular value such that the preimage of every critical value contains a single critical point. We may now also assume that none of the points of $S$ lie in the preimages of the critical points of $\rho$ and that $S \cap \{\rho = 0\} = \varnothing$. Recall that $f$ restricts to an embedding of an open neighbourhood of $K_0 = \{\rho \le 0\}$.

As before, we choose an increasing sequence of regular values $0 = c_0 < c_1 < c_2 < \dots$ of $\rho$ with $\lim\limits_{j \to \infty}c_j = \infty$ satisfying $c_{2j-1} < \rho(p_j) < c_{2j}$ for $j = 1,2,\dots$, where $p_j$ are the critical points of $\rho$ outside $K_0$. Set $K_j = \{\rho \le c_j\}$, $j = 1,2,\dots$. Here we may assume that each pair $c_{2j-1}$ and $c_{2j}$ is sufficiently close to $\rho(p_j)$ that $S \cap (K_{2j}\setminus \mathring{K}_{2j-1}) = \varnothing$ for all $j = 1,2,\dots$. That is, the points of $S$ outside $K_0$ only occur in the sets $\mathring{K_{j}} \setminus K_{j-1}$ when $K_j$ has the same topology as $K_{j-1}$ (a non-critical step, in our present terminology). Since $S$ is discrete, each such set $\mathring{K_{j}} \setminus K_{j-1}$ contains finitely many points of $S$.

Let $S_j = S \setminus K_{j-1}$, $j = 1, 2, \dots$, be the points of $S$ outside the compact set $K_{j-1}$. Since $f\rvert_S : S \to (\C^*)^2$ is a proper map, we may choose an increasing sequence $r_j \in \R$ with $\lim\limits_{j\to\infty} r_j = \infty$ such that $f(S_j) \cap \overline{P}_{r_j} = \varnothing$ for all $j = 1,2,\dots$ (recall that $P_r = \varnothing$ for $r \le 0$).

Setting $K_{-1} = \{\rho \le c_{-1}\}$ for some regular value $c_{-1} < 0$ close to $0$, so that $K \subset \mathring{K}_{-1}$, we now construct as before sets $L_j$, smooth diffeotopies $\phi_{j,t}$ of $X$, and homotopies of continuous maps $f_{j,t} : X \to (\C^*)^2$ so that conditions (\ref{proof:embedding_approx:deform_compact})--(\ref{proof:embedding_approx:hol_embedding_final}) given earlier hold. We replace conditions (\ref{proof:embedding_approx:boundary_far_away}) and (\ref{proof:embedding_approx:collar_far_away}) by the following two conditions (\ref{proof:embedding_approx:boundary_far_away_S}$'$) and (\ref{proof:embedding_approx:collar_far_away_S}$'$), and add two additional conditions (\ref{proof:embedding_approx:interpolation_S}$'$) and (\ref{proof:embedding_approx:identity_diffeo_S}$'$). 
\begin{enumerate}[(i$'$)]
 \setcounter{enumi}{7}
 \item \label{proof:embedding_approx:boundary_far_away_S} $f_{j,1}(bL_j) \subset (\C^*)^2\setminus \overline{P}_{r_j}$.
 \item \label{proof:embedding_approx:collar_far_away_S} $f_{j,1}(L_j \setminus \mathring{L}_{j-1}) \subset (\C^*)^2 \setminus \overline{P}_{r_{j-1}}$.
 \item \label{proof:embedding_approx:interpolation_S} $f_{j,t}\rvert_S = f\rvert_S$ for all $t \in I$.
 \item \label{proof:embedding_approx:identity_diffeo_S} There exists a neighbourhood of $K \cup S$ on which $\phi_{j,t}$ restricts to the identity for all $t \in I$.
\end{enumerate}

Assuming this done, we may repeat the argument given previously, using condition (\ref{proof:embedding_approx:collar_far_away_S}$'$) to give properness of the limit map, and conditions (\ref{proof:embedding_approx:matching}), (\ref{proof:embedding_approx:interpolation_S}$'$) and (\ref{proof:embedding_approx:identity_diffeo_S}$'$) to obtain conclusions (\ref{thm:embedding_approx:inclusion}) and (\ref{thm:embedding_approx:interpolation}) of the theorem. The full theorem is proved.

Choosing values $r_{-1} < r_0 < 0$, so that $P_{r_{-1}} = P_{r_0} = \varnothing$, the initial case $j = 0$ is unchanged from before. We thus fix some value $j > 0$ and suppose the conditions are satisfied for $j-1$. We also suppose that there are no critical points of $\rho$ in the set $K_j \setminus \mathring{K}_{j-1}$, so that we are in the non-critical case. We explain how to modify the construction previously given for $S = \varnothing$ to ensure interpolation at the points of $S \cap (K_j \setminus \mathring{K}_{j-1})$ without changing the values $f$ takes at all other points of $S$.

We repeat the argument given previously to obtain a compact bordered Riemann surface $L_j$ whose interior contains $L_{j-1}$, a diffeotopy $\sigma_{j,t}$ of $X$ satisfying $\sigma_{j,1}(L_j) = L_{j-1}$ with $\sigma_{j,t}$ equal to the identity on $\cup_{t_0 \le t \le 1}\phi_{j-1,t}(K_{j-2})$ for some $t_0 < 1$ close to $1$, and a homotopy $f_{j,t} : X \to (\C^*)^2$ linking $f_{j-1,1} = f_{j,0}$ to some continuous map $f_{j,1} : X \to (\C^*)^2$ that together satisfy conditions (\ref{proof:embedding_approx:approximation}), (\ref{proof:embedding_approx:hol_embedding}), (\ref{proof:embedding_approx:hol_embedding_final}), (\ref{proof:embedding_approx:boundary_far_away_S}$'$), and (\ref{proof:embedding_approx:collar_far_away_S}$'$). Note that since $S \cap bK_{j-1} = \varnothing$, and since $\phi_{j-1,t}$ is the identity on a neighbourhood of $S$ by the inductive assumption, we also have $S \cap bL_{j-1} = S \cap \phi_{j-1,1}(bK_{j-1}) = \varnothing$. By Lemma \ref{lem:enlargingdomainofembedding} we may therefore ensure that $\sigma_{j,t}$ equals the identity on a neighbourhood of $S$, and that $f_{j,t}\rvert_S = f_{j-1,1}\rvert_S = f\rvert_S$ for all $t \in I$.

In the same way as before, we construct a diffeotopy $\psi_{j,t}$ of $X$ such that $\psi_{j,1}(K_j) = K_{j-1}$, and such that $\psi_{j,t}$ is the identity on a certain compact subset of $\mathring{K}_{j-1}$. Here we may also assume that $\psi_{j,t}$ is the identity on a neighbourhood of the set $S \setminus (K_j \setminus \mathring{K}_{j-1})$.

Recall that there are only finitely many points of $S$ in $K_j \setminus \mathring{K}_{j-1}$, and that in fact they all lie in the open set $\mathring{K}_j \setminus K_{j-1}$. Let $s \in S \cap (\mathring{K}_j \setminus K_{j-1})$. Then $t \mapsto \psi_{j,t}(s)$, $t \in I$, is a smoothly embedded curve in $X$ starting at $\psi_{j,0}(s) = s$ and ending at some point in $\mathring{K}_{j-1}$. Let $0 < t' < 1$ be the unique value such that $s' := \psi_{j,t'}(s) \in bK_{j-1}$. Then $\eta(t) := \psi^{-1}_{j,t}(s')$, $t \in I$, is a smoothly embedded curve in $X$ passing through $s$, starting at $s' \in bK_{j-1}$ and ending at $s'' := \psi^{-1}_{j,1}(s') \in bK_j$. Note that there could be additional points of $S$ on the curve $\eta$ besides $s$. Let $S'$ be the set of all such points, including $s$ itself.

Let $\lambda : (-1,1) \to bK_{j-1}$ be a diffeomorphism onto a neighbourhood of $s'$ in $bK_{j-1}$ such that $\lambda(0) = s'$. Because $\psi_{j,t}$ is given by the flow of a complete smooth vector field $Y$ on $X$ supported in a neighbourhood of $K_j \setminus \mathring{K}_{j-1}$, it is actually defined for all $t \in \R$. For any $0 < \delta < 1$ we may therefore construct a diffeomorphism $\theta : (-\delta,\delta) \times (-\delta, 1 + \delta) \to W_\delta$ onto a neighbourhood $W_\delta \subset X$ of the image of $\eta$ by
\[
 \theta(x,t) = \psi^{-1}_{j,t}(\lambda(x))\,,\quad x \in (-\delta,\delta), t \in (-\delta, 1 + \delta)\,.
\]

Let $\tilde\chi : X \to [0,1]$ be a smooth cutoff function with support in $W_\delta$ that equals $1$ in a neighbourhood of the image of $\eta$, and let $\chi = 1 - \tilde\chi$. Consider the smooth vector field $\tilde Y = \chi\cdot Y$ on $X$. Clearly, $\tilde Y$ is also complete and its flow $\tilde\psi_{j,t}$ equals the identity on a neighbourhood of the curve $\eta$. Define $\phi_{j,t} : X \to X$ by the formula (\ref{formula:diffeotopy}) given earlier. If we now replace $\psi_{j,t}$ by $\tilde\psi_{j,t}$ in the definition of $\phi_{j,t}$ we obtain a diffeotopy $\tilde\phi_{j,t}$ of $X$ such that $\tilde\phi_{j,t} = \phi_{j,t}$ on $K_j \setminus W_\delta$ for all $t \in I$. Note that $\tilde\phi_{j,t}$ is the identity in a neighbourhood of each point of $S'$. We see that $\tilde\eta := \tilde\phi_{j,1}\circ\eta$ is a curve starting at $\tilde s' := \tilde\phi_{j,1}(s') = \sigma^{-1}_{j,1}(\phi_{j-1,1}(\tilde\psi_{j,1}(s'))) \in bL_j$ (since $\tilde\psi_{j,1}(s') = s' \in bK_{j-1}$ and $\phi_{j-1,1}(bK_{j-1}) = bL_{j-1}$) that passes through the points of $S'$.

Note that $f_{j,1}(\tilde s') \in (\C^*)^2 \setminus \overline{P}_{r_j}$ since $\tilde s' \in bL_j$, and that $f_{j,1}(S') \subset (\C^*)^2 \setminus \overline{P}_{r_j}$ by choice of $r_j$.
Deform $f_{j,1}$ on a neighbourhood of $\tilde\eta$, keeping it fixed in a neighbourhood of $\tilde s'$ and at the points of $S'$, to obtain a map $f'$ on $X$ that is still a holomorphic embedding on a neighbourhood of $L_j$, and is now also a smooth embedding on $\tilde\eta$ that satisfies $f'(\tilde\eta) \subset (\C^*)^2 \setminus \overline{P}_{r_j}$.

Using Mergelyan's theorem, approximate $f'$ uniformly in a neighbourhood of $L_j$ and in the $\CC^1$-topology on $\tilde\eta$ by a holomorphic embedding $\tilde f'$ defined on a smaller neighbourhood of $L_j \cup \tilde\eta$. Ensure that the values of $f'$ at the points of $S'$ and at all points of $S$ inside $L_j$ do not change during this process. If the approximation is sufficiently good then $f'$ is homotopic to $\tilde f'$ in a neighbourhood of $L_j \cup \tilde\eta$ through a family of maps that are holomorphic embeddings of a neighbourhood of $L_j$. We now patch this local homotopy to $f'$ using a suitable cutoff function in the homotopy parameter, thereby giving a homotopy from $f'$ to some map $f''$ that agrees with $\tilde f'$ on a neighbourhood of $L_j \cup \tilde\eta$. Combining these homotopies with the homotopy $f_{j,t}$ and reparameterising gives a homotopy from $f_{j,0}$ to $f''$ defined on all of $X$. We continue to call this new homotopy $f_{j,t}$, so that now $f'' = f_{j,1}$.

The set $L'_j := \tilde\phi_{j,1}(K_j)$ is a compact bordered Riemann surface whose boundary agrees with that of $L_j$ everywhere except near $\tilde\eta$. Taking $\delta$ sufficiently small we can ensure that $L'_j$ is contained in the neighbourhood of $L_j \cup \tilde\eta$ where $f'' = f_{j,1}$ is a holomorphic embedding, that $f_{j,1}(bL'_j) \subset (\C^*)^2\setminus\overline{P}_{r_j}$, and that $f_{j,1}(L'_j\setminus \mathring{L}_{j-1}) \subset (\C^*)^2\setminus\overline{P}_{r_{j-1}}$. Redefining $L_j = L'_j$ and $\phi_{j,t} = \tilde\phi_{j,t}$, we are done. By repeating this argument finitely many times for the points of $S$ in $K_j\setminus \mathring{K}_{j-1}$, we can ensure all conditions are satisfied. This completes the argument for the non-critical case.

Since the points of $S$ only occur in non-critical steps in the induction process, the critical case is the same as before, except that we must keep the values of $f$ fixed at all points of $S$ throughout the construction. It is clear that this additional requirement can be incorporated.
\end{proof}

\end{document}